\documentclass{amsart}

\usepackage[utf8]{inputenc}
\usepackage[T1]{fontenc}
\usepackage{textcomp}
\usepackage{amsfonts}
\usepackage{amsthm}
\usepackage{amssymb}
\usepackage{pst-node}
\usepackage[all]{xy}
\usepackage{mathtools}
\usepackage{chngcntr}
\usepackage{thmtools}
\usepackage{mathrsfs}
\usepackage{tcolorbox}
\usepackage{listings}
\usepackage{imakeidx}
\usepackage{hyperref}
\usepackage{color}
\usepackage[usenames,dvipsnames]{xcolor}
\usepackage{float}
\usepackage{tikz}
\usepackage{enumerate}
\usepackage[nameinlink, capitalize, noabbrev]{cleveref}
\usepackage{natbib}
\usepackage{caption}
\usetikzlibrary{intersections}
\usepackage{pgfplots}
\pgfplotsset{width=10cm,compat=1.9}
\usepgfplotslibrary{external}

\hypersetup{ colorlinks=true,linkcolor=teal,    citecolor=magenta, }

\def\set#1{{\def\st{\;:\;}\left\{#1\right\}}}
\def\abs#1{\left\vert{#1}\right\vert}
\def \<#1>{{\left\langle{#1}\right\rangle}}

\def\mat#1{\left(\begin{matrix}#1\end{matrix}\right)}

\def\ZZ{\mathbb Z}
\def\QQ{\mathbb Q}
\def\PP{\mathbb P}

\def\FF{\mathbb F}
\def\CC{\mathbb C}
\def\NN{\mathbb N}
\def\RR{\mathbb R}
\def\Q-{\overline{\mathbb Q}}

\DeclareMathOperator{\Gal}{Gal}
\DeclareMathOperator{\Aut}{Aut}
\DeclareMathOperator{\FPP}{FPP}

\DeclareMathOperator{\St}{St}
\DeclareMathOperator{\Stab}{Stab}
\DeclareMathOperator{\orb}{orb}
\DeclareMathOperator{\Sym}{Sym}

\DeclareMathOperator{\st}{st}
\DeclareMathOperator{\IIm}{Im}
\DeclareMathOperator{\GL}{GL}

\newtheorem{Theorem}{Theorem}

\newtheorem{Corollary}[Theorem]{Corollary}
\newtheorem{Proposition}{Proposition}[section]
\newtheorem{Lemma}[Proposition]{Lemma}
\newtheorem{corollary}[Proposition]{Corollary}
\newtheorem{theorem}[Proposition]{Theorem}

\newtheorem{Question}[Proposition]{Question}
\newtheorem{Definition}[Proposition]{Definition}

\newtheorem{Remark}[Proposition]{Remark}

\title[Branch iterated Galois groups with positive fixed-point proportion and positive Hausdorff dimension]{Branch iterated Galois groups with positive fixed-point proportion and positive Hausdorff dimension}
\author{Santiago Radi}
\address{Santiago Radi: Department of Mathematics, Texas A\&M University, 77843 College Station, U.S.A.}
\email{santiradi@tamu.edu}
\thanks{The author is supported by Grigorchuk's Simons Foundation Grant MP-TSM-00002045 and the department of Mathematics of Texas A\&M University.}
\date{October 2025}

\begin{document}

\begin{abstract}
In this article we prove that the arithmetic profinite iterated monodromy group of a post-critically infinite unicritical polynomial is regular branch (and so of positive Hausdorff dimension), and has positive fixed-point proportion when the degree is odd. The examples are instances of a bigger family of regular branch groups constructed in this article, whose fixed-point proportion can be computed explicitly and is positive in many cases. 
This gives the first examples outside the binary rooted tree where a level-transitive group has positive Hausdorff dimension and positive fixed-point proportion, answering in the negative a question of Jones (2008).
\end{abstract}

\maketitle

\section{Introduction}
\label{section: Introduction}

Given a $d$-regular rooted tree $T$, the group of automorphisms of $T$, denoted $\Aut(T)$, is the set of bijective functions $g: T \rightarrow T$ that preserve adjacency between the vertices. Since elements in $\Aut(T)$ preserve adjacency, the group $\Aut(T)$ acts on $\mathcal{L}_n$, the $n$-th level of the tree for each $n \geq 0$. The kernel of each of these actions is denoted $\St(n)$ and then one has natural projections $\pi_n: \Aut(T) \rightarrow \Aut(T)/\St(n)$, which corresponds to the action of the elements in $\Aut(T)$ over the first $n$ levels. If $G$ is a subgroup of $\Aut(T)$, then $\St_G(n) := G \cap \St(n)$ is the stabilizer of level $n$ of $G$ and we have the restrictions of the maps $\pi_n: G \rightarrow G/\St_G(n)$. The fixed-point proportion of $G$ is calculated as $$\FPP(G) = \lim_{n \rightarrow \infty} \frac{\# \set{g \in \pi_n(G): \text{$g$ fixes a vertex in $\mathcal{L}_n$}}}{\# \pi_n(G)}.$$

Odoni in \cite{Odoni1985} was the first one to realize that this quantity had applications to problems in arithmetic dynamics. Using the fixed-point proportion, Odoni proved that the Dirichlet density of primes dividing at least one term in the Sylvester sequence is zero, where the Sylvester sequence is given by the initial term $w_1 =2$ and the recursion $w_n = 1+w_1\cdots w_{n-1}$.

Odoni’s ideas were later generalized to number fields and  rational functions. Given $K$ a number field and $f$ a rational function with degree at least $2$ and coefficients in $K$, the extension $K(\cup_{n \geq 0} f^{-n}(0))/K$ is Galois and its Galois group acts faithfully on the $d$-regular rooted tree of preimages of $0$ (see \cite{Jones2014}). Consequently, it embeds in $\Aut(T)$ as a closed subgroup. This group is denoted $G_\infty(K,f,0)$ and called the iterated Galois group. It is worth mentioning that the notation and the name of this group has not been standardized yet and the same object can be found with names as dynamical Galois group, arboreal Galois group or the image of the arboreal representation. 

If $a_0$ is an element of $K$, we have that the Dirichlet density of the set $$P_f(a_0) := \set{\mathfrak{p} \in M_K^0: v_\mathfrak{p}(f^n(a_0)) > 0 \text{ for some $n$ such that $f^n(a_0) \neq 0, \infty$}},$$ is bounded above by $\FPP(G_\infty(K,f,0))$ (see \cite[Theorem 6.1]{JonesManes2012}). Here, $M_K^0$ denotes the set of prime ideals of the ring of integers of $K$. 

Later, more applications of the fixed-point proportion of iterated Galois groups to problems in arithmetic dynamics were discovered (see for example \cite{BridyJones2022, Juul2014}). Replacing $K$ by any field and $0$ by an element $t$ transcendental over $K$, the extension $K(\cup_{n \geq 0}f^{-n}(t))/K(t)$ is still Galois and one denotes the Galois group of the extension as $G_\infty(K,f,t)$. This group also acts faithfully in the tree of preimages of $t$, it embeds as a closed subgroup in $\Aut(T)$ and it represents the generic iterated Galois group. In a loose sense, $G_\infty(K,f,t)$ is the Galois group one expects when we specialize $t$ to a value in $K$. The group $G_\infty(K,f,t)$ will be called arithmetic iterated Galois group in this paper, although it can also be referred as arithmetic iterated monodromy group by other authors. The group is level-transitive, namely, the action of $G_\infty(K,f,t)$ over $\mathcal{L}_n$ is transitive for every $n \in \NN$.

In the case studied by Odoni, the iterated Galois group was $\Aut(T)$ and it is known that the fixed-point proportion of $\Aut(T)$ is zero (see for example \cite[Corollary 2.6]{AbertVirag2004}). Even more, if the iterated Galois group has finite index with respect to $\Aut(T)$, the fixed-point proportion is also zero (see \cref{lemma: FPP first properties}.(1)). However, it is not completely clear yet in which cases the iterated Galois group has finite index and many cases where the index is not finite are known (see \cite[Section 3]{Jones2014} for a discussion about this and \cite[Conjecture 3.11]{Jones2014} for a conjecture in the quadratic case).  

Although there are examples with infinite index, there is a tendency for the iterated Galois groups to have positive Hausdorff dimension. Given $G \leq \Aut(T)$, the Hausdorff dimension of $G$ is defined as $$\mathcal{H}(G) := \liminf_{n \rightarrow +\infty} \frac{\log(\abs{\pi_n(G)})}{\log(\abs{\pi_n(\Aut(T))})}.$$

The Hausdorff dimension is another way to measure how large a group inside $\Aut(T)$ is. In the context of iterated Galois groups, Pink proved in \cite[Theorem 4.4.2]{Pink2013} that if $f$ is a quadratic rational function with critical points $c_1$ and $c_2$ whose post-critical orbits are infinite, and $r$ is the smallest natural number such that $f^{r+1}(c_1) = f^{r+1}(c_2)$, then $\mathcal{H}(G_\infty(K,f,t)) = 1- \frac{1}{2^r}$, where $K$ is any field of characteristic different to $2$ and $t$ is transcendental over $K$. In case of degree greater than $2$, authors proved in \cite[Theorem 1.1]{Benedetto2017} that the iterated Galois group of the polynomial $3x^2-2x^3$ has Hausdorff dimension $\approx 0.871$. However, both cases are known to have null fixed-point proportion. On the other hand, Jones proved in \cite{Jones2012} that if $T_d$ corresponds to the Chebyshev polynomial of degree $d$ and $t$ is transcendental over $\CC$, then 
\begin{align*}
\FPP(G_\infty(\CC,\pm T_d,t)) = \left \{ \begin{matrix} 
1/2 & \mbox{if $d$ is odd,} \\ 
1/4 & \mbox{if $d$ is even.}
\end{matrix}\right.
\end{align*} 
However, the Hausdorff dimension of $G_\infty(\CC,\pm T_d, t)$ is zero (the group is isomorphic to the infinite dihedral group and therefore satisfies the group law $[x^2,y^2] = 1$. Then the Hausdorff dimension is zero by \cite[Corollary 5]{Fariña2025}). 

This led Jones to pose the following conjecture:

\begin{Question}[{\cite[Conjecture 17]{Jones2008survey}}]
If $G$ is a level-transitive group in $\Aut(T)$ with $\mathcal{H}(G) > 0$, then $\FPP(G) = 0$.
\label{Conjecture: Jones hdim > 0 fpp = 0}
\end{Question}

\cref{Conjecture: Jones hdim > 0 fpp = 0} was addressed by Boston in \cite{Boston2010}, who constructed a group with Hausdorff dimension $1/2$ and fixed-point proportion of at least $1/4$ acting on the binary rooted tree. Nevertheless, the exact value of the fixed-point proportion is not provided nor is it proved that the group correspond to an iterated Galois group. For larger trees, there are no examples constructed.

In light of Boston's counterexample, one can modify \cref{Conjecture: Jones hdim > 0 fpp = 0} by asking if the result holds when we restrict to iterated Galois groups. Even more, one can restrict further to the case where the iterated Galois groups are branch. A group $G \leq \Aut(T)$ is branch if it is level-transitive and for every level $n$, the elements of $G$ that act non-trivially below only one subtree at level $n$ generate a subgroup of finite index. From the point of view of number theory, branchness is related to a rapid growth of the field extensions of any preimage of $t$ when we take its iterations by $f$. An important subfamily of closed branch groups are regular branch groups, which can be given in terms of a finite set of allowed actions (called set of patterns) and have recently been studied in \cite{RadiFiniteType2025}. By \cite[Proposition 2.7]{Bartholdi2006}, closed regular branch groups have positive Hausdorff dimension and this dimension is computable in terms of its set of patterns.

The existence of regular branch iterated Galois groups is known. For example, combining the results of Pink (see \cite[Theorem 6.4.2]{Self_similar_groups}) and the authors in \cite[Theorem 2.5]{GrigSavchukSunic2007}, one has that $G_\infty(\CC, z^2+i,t)$ is regular branch. However, combining the results of the author in \cite[Section 8.3]{RadiFiniteType2025} and \cite[Theorem 1]{FariñaRadi2025}, one obtains that the fixed-point proportion of $G_\infty(\CC,z^2+i,t)$ is zero. 

\begin{Question}
Can we find regular branch (and therefore with positive Hausdorff dimension) iterated Galois groups with positive fixed-point proportion?
\label{Question: positive FPP Hausdorff dimension}
\end{Question}

In this article, we answer \cref{Question: positive FPP Hausdorff dimension} in the positive by providing examples of regular branch iterated Galois groups whose fixed-point proportion and Hausdorff dimension are calculated explicitly. The main result is as follows:

\begin{Theorem}
Let $d \geq 2$, a field $K$ such that the characteristic of $K$ does not divide $d$ and $K^{sep}$ a separable closure of $K$. Let $t$ be a transcendental element over $K$, an element $c \in K$ and $\xi \in K^{sep}$ a primitive $d$th root of unity. Suppose that $f(x) = x^d + c \in K[x]$ is a post-critically infinite unicritical polynomial. Then, the group $G_\infty(K,x^d+c,t)$ is regular branch, its Hausdorff dimension is $$\mathcal{H}(G_\infty(K,x^d+c,t)) = \frac{\log(d)}{\log(d!)}$$ and its fixed-point proportion is $$\FPP(G_\infty(K,x^d+c,t)) = \frac{\# \set{a \in I: a-1 \in (\ZZ/d\ZZ)^\times}}{\Phi(d)},$$ \\ where $I$ is the image of the map $\Gal(K(\xi)/K) \rightarrow (\ZZ/d\ZZ)^\times$ given by $\sigma \mapsto a$ with $\sigma(\xi) = \xi^a$, and $\Phi$ is the Euler's totient function. 

In the particular case where $K = \QQ$, we have $$\FPP(G_\infty(\QQ,x^d+c,t)) = \prod_{p \mid d} \frac{p-2}{p-1},$$ where the product runs over the prime divisors of $d$. So, if $d$ is odd, the fixed-point proportion is positive.
\label{theorem: FPP x^d+1}
\end{Theorem}

The iterated Galois group $G_\infty(K,x^d+c,t)$ with $t$ transcendental over $K$ has recently been studied by Adams and Hyde in \cite{AdamsHyde2025}. \cref{theorem: FPP x^d+1} allow us to have information about its fixed-point proportion. It is worth mentioning that in order to have positive fixed-point proportion we take advantage that the field is not separably closed. The case of polynomials where $K$ is separably closed was completely classified in \cite{FariñaRadi2026FPP} and the only polynomials with positive fixed-point proportion are conjugate to Chebyshev polynomials. 

The proof of \cref{theorem: FPP x^d+1} is based on the construction of a bigger family of regular branch groups whose fixed-point proportion can be calculated explicitly. Chosen $d \geq 2$ and two subgroups $\mathcal{Q}, \mathcal{P} \leq \Sym(d)$ such that $\mathcal{Q} \lhd \mathcal{P}$, we define in the $d$-regular rooted tree $T$ the subgroup $G_\mathcal{Q}^\mathcal{P} \leq \Aut(T)$ consisting of all the elements in $\Aut(T)$ whose action over the children of any vertex is in the same coset of $\mathcal{P}/\mathcal{Q}$. Using another instance of this family, we provide more counterexamples to \cref{Conjecture: Jones hdim > 0 fpp = 0} in the case of $d$-regular rooted trees with $d \neq 2 \pmod{4}$:

\begin{Theorem}
Let $d \geq 2$ and write $d = 2^n r$ with $r$ odd. Take $\mathcal{Q} = C_2^{n} \times C_r$ and $\mathcal{P} = \mathcal{Q} \rtimes \Aut(\mathcal{Q})$. Then $\mathcal{Q}$ and $\mathcal{P}$ can be embedded in $\Sym(d)$ such that $\mathcal{Q}$ acts transitively on $\set{1,\dots,d}$ and consequently $G_\mathcal{Q}^\mathcal{P}$ is regular branch, its Hausdorff dimension equals $$\mathcal{H}(G_\mathcal{Q}^\mathcal{P}) = \frac{\log(d)}{\log(d!)}$$ and the fixed-point proportion is $$\FPP(G_\mathcal{Q}^\mathcal{P}) = \frac{\# \set{A \in \GL_n(\FF_2): A-1 \in \GL_n(\FF_2)}}{\abs{\GL_n(\FF_2)}} \, \prod_{p \in \PP, p \mid r} \left( \frac{p - 2}{p - 1} \right).$$
When $d \neq 2 \pmod{4}$, the fixed-point proportion is positive.
\label{theorem: main theorem construction 2}
\end{Theorem}

The key fact that allows one to calculate explicitly the fixed-point proportion of the groups $G_\mathcal{Q}^\mathcal{P}$ comes from a generalization of iterated wreath products that we do in this article. If $S$ is a subset of $\Sym(d)$, the iterated wreath product $W_S$ is the subset of elements in $\Aut(T)$ whose action on the children of any vertex of $T$ is an element in $S$. 

In the case that $S = H$ is a subgroup of $\Sym(d)$, then $W_H$ is a closed subgroup of $\Aut(T)$. In \cite[Corollary 2.6]{AbertVirag2004}, authors proved that when $H$ is transitive over $\set{1,\dots,d}$, then $\FPP(W_H) = 0$.

Extending the notion of fixed-point proportion to sets of $\Aut(T)$ (see \cref{equation: definition FPP for sets} for a definition), we obtain an explicit way to calculate the fixed-point proportion of any iterated wreath product $W_S$:

\begin{Theorem}
Let $d \geq 2$ and $S$ any set in $\Sym(d)$. Given $k \in \set{0, \dots, d}$, let $D_S(k)$ be the number of permutations in $S$ that fix exactly $k$ points of $\set{1,\dots,d}$ and define the polynomial $f_S(x) := \sum_{k = 1}^d \frac{D_S(k)}{\# S} (1 - (1-x)^k)$. Then
\begin{enumerate}
    \item $\FPP(W_S)$ equals the largest fixed point of $f_S$ in $[0,1]$.
    \item $\FPP(W_S)$ is the solution of a polynomial of degree at most $d-1$ with coefficients in the ring $\ZZ\left[\frac{1}{\# S}\right]$. In particular, it is an algebraic integer.
    \item $\FPP(W_S) = 0$ if and only if $f_S'(0) \leq 1$ and $f_S$ is not the identity function.
    \item $\FPP(W_S) = 1$ if and only if every element in $S$ fixes at least one element in $\set{1,\dots,d}$.
\end{enumerate}
\label{theorem: FPP(WS) in introduction}
\end{Theorem}

In the particular case where $S = H$ is a subgroup of $\Sym(d)$, \cref{theorem: FPP(WS) in introduction} extends the result in \cite[Corollary 2.6]{AbertVirag2004} for any subgroup of $\Sym(d)$, either transitive or not:

\begin{Corollary}
If $d \geq 2$, $H$ is a subgroup of $\Sym(d)$ and $X = \set{1,\dots,d}$. Then, the fixed point proportion of $W_H$ is: 
\begin{align*}
\FPP(W_H) = \left \{ \begin{matrix} 
0 & \mbox{if $\mathcal{P}$ is transitive over $X$,} \\ 
\alpha \in (0,1) & \mbox{if $\mathcal{P}$ not transitive and $\exists \, \sigma$ with no fixed points} \\
1 & \mbox{if $\mathcal{P}$ not transitive and every element fixes a point,}
\end{matrix}\right.
\end{align*}
where $\alpha$ is an algebraic integer.
\label{Corollary: FPP iterated wreath products}
\end{Corollary}

\subsection*{Organization}

This article is organized into seven sections. \cref{section: preliminaries} introduces the background necessary to follow this article. \cref{section: FPP of IWP} focuses on proving \\ \cref{theorem: FPP(WS) in introduction} and \cref{Corollary: FPP iterated wreath products}. In \cref{section: the main construction}, we define the groups $G_\mathcal{Q}^\mathcal{P}$ and characterize when they are level-transitive, calculate their Hausdorff dimensions and their fixed-point proportions. \cref{section: examples} includes two examples, with the second one resulting in \cref{theorem: main theorem construction 2}. In \cref{section: Application to iterated Galois groups}, it is proved that the first example in \cref{section: examples} corresponds to the iterated Galois group of $x^d+c$, proving \cref{theorem: FPP x^d+1}. Finally, in \cref{section: open questions}, the author presents open questions considered relevant to the theory of fixed-point proportions.

\subsection*{Acknowledgements}

The idea for this work emerged during the workshop Groups of Dynamical Origin in Pasadena, 2024, organized by R. Grigorchuk, D. Savchuk, and C. Medynets \cite{Dynamicalorigin2024}. The author is grateful to the American Institute of Mathematics (AIM) for sponsoring the workshop, which facilitated the discussion of new problems at the intersection of group theory, number theory, and arithmetic dynamics. The author also thanks all the participants of the conference who worked on the fixed-point proportion project, particularly V. Goksel, who suggested the problem addressed in this article. Gratitude is also extended to T. Tucker, who, in a private communication, provided an example that motivated the construction of this family of groups, and to J. Fariña-Asategui for reading and suggesting corrections. Finally, to thank the anonymous referee, whose feedback contributed greatly to the final version of this paper.

\section{Preliminaries}
\label{section: preliminaries}

\subsection{About general notation} 

Given a finite set $S$, the notation $\#S$ will refer to its number of elements. In the case where $S$ is a subgroup, we will use $\abs{S}$ instead. If $H$ is a subgroup of $G$, we will use $H \leq G$ and if $H$ has finite index in $G$, we will denote it by $H \leq_f G$. The subgroup $G'$ will denote the commutator subgroup of $G$ and $N_G(H)$ the normalizer of $H$ in $G$. 

\subsection{Groups acting on rooted trees}
\label{subsec_Groups_acting_on_trees}

A \textit{$d$-regular rooted tree} $T$ is a tree with infinitely many vertices and a root $\emptyset$, where all the vertices have the same number of descendants $d \geq 2$. The set of vertices at a distance exactly $n\ge 0$ from the root form the $n$-th level of $T$ and it will be denoted $\mathcal{L}_n$. The vertices whose distance is at most $n$ from the root form the $n$-th truncated tree of $T$, denoted $T^n$. The group of automorphisms of $T$, denoted $\Aut(T)$, is the group of bijective functions from $T$ to $T$ that send the root to the root and preserve adjacency between the vertices. This in particular implies that $\Aut(T) \curvearrowright \mathcal{L}_n$ for all $n \in \NN$. For any vertex $v\in T$, the subtree rooted at $v$, which is again a $d$-regular rooted tree, is denoted $T_v$. In this article, the action of $\Aut(T)$ on $T$ will be on the left, so if $g,h \in \Aut(T)$ and $v \in T$, then $(gh)(v) = g(h(v))$. 

Given a vertex $v \in T$ we write $\st(v)$ for the \textit{stabilizer of the vertex $v$}, namely, the subgroup of the elements $g \in \Aut(T)$ such that $g(v) = v$. Given $n \in \NN$, we write $\St(n) = \bigcap_{v\in \mathcal{L}_n}\st(v)$, and we call it the \textit{stabilizer of level $n$}. Then $\St(n)$ is a normal subgroup of finite index and $\Aut(T)/\St(n)$ is isomorphic to $\Aut(T^n)$. We can make $\Aut(T)$ a topological group by declaring $\St(n)$ to be a base of neighborhoods of the identity. Then $\Aut(T)$ is homeomorphic and isomorphic to $\varprojlim \Aut(T)/\St(n)$, making $\Aut(T)$ a profinite group. The transition maps will be $\pi_n: \Aut(T) \rightarrow \Aut(T^n)$ corresponding to the restricting of the action of an element on the whole tree to only the first $n$ levels. 

Let $v \in T$ and $g \in \Aut(T)$. By preservation of adjacency, we have $g(T_v) = T_{g(v)}$. Calling $g|_v$ the restriction of the map $g$ to $T_v$, for all $w \in T_v$ we have $$g(vw)=g(v)g|_v(w).$$ This map is called the \textit{section} of $g$ at vertex $v$. To ease notation, we will write $g|_v^n$ for $\pi_n(g|_v)$. A way to describe an element $g$ in $\Aut(T)$ is by decorating each vertex $v$ on the tree with the permutation $g|_v^1$. Notice that $g|_v^1$ is a permutation belonging to $\Sym(d)$. We will call it \textit{label} and the collection of all the labels of an element $g$ completely determines $g$ and is called the \textit{portrait} of $g$. The sections satisfy the following two properties: 
\begin{align}
(gh)|_v^n = g|_{h(v)}^n h|_v^n \text{ and } (g^{-1})|_v^n = (g|_{g^{-1}(v)}^n)^{-1}.
\label{equation: properties sections}
\end{align}

Let us fix now a subgroup $G \le \Aut(T)$. Then $G$ also has actions on $T$ and $\mathcal{L}_n$. We define vertex stabilizers and level stabilizers by restricting the previous ones to the group, i.e, $\st_G(v) := \st(v) \cap G$ and $\St_G(n) := \St(n) \cap G$ for $v \in T$ and $n \in \NN$.

\begin{Definition}
We say that a group $G \leq \Aut(T)$ is \textit{level-transitive} if the actions $G \curvearrowright \mathcal{L}_n$ are transitive for all $n \in \NN$. 
\end{Definition}

\begin{Definition}
Consider $H \leq G \leq \Aut(T)$. The \textit{relative Hausdorff dimension} of $H$ on $G$ is defined as the number $$\mathcal{H}_G(H) := \liminf_{n \rightarrow +\infty} \frac{\log(\abs{\pi_n(H)})}{\log(\abs{\pi_n(G)})}.$$
\end{Definition}
\begin{Lemma}
Let $T$ be a $d$-regular rooted tree and $K \leq_f H \leq G \leq \Aut(T)$. Then
\begin{enumerate}
    \item $\mathcal{H}_G(H) = \mathcal{H}_G(\overline{H})$,
    \item if $G$ has infinitely many elements, then $\mathcal{H}_G(H) = \mathcal{H}_G(K)$.
\end{enumerate}
\label{lemma: Hausdorff dimension first properties}
\end{Lemma}

\begin{proof}
(1) follows from the fact that $\pi_n(H) = \pi_n(\overline{H})$ for all $n \in \NN$ and (2) because for $n$ big enough we have $\abs{\pi_n(H)} = [H:K] \abs{\pi_n(K)}$.
\end{proof}

\begin{Definition}
If $G \leq \Aut(T)$, the \textit{Hausdorff dimension} of $G$ is defined as $$\mathcal{H}(G) := \mathcal{H}_{\Aut(T)}(G).$$
\end{Definition}

\subsection{Groups of finite type}
\label{section: Groups of finite type}

We say that a group $G \leq \Aut(T)$ is \textit{self-similar} if $g|_v \in G$ for all $v \in T$ and $g \in G$. 

If $K$ is a subgroup of $\Aut(T)$ and $n \in \NN$, we define the \textit{geometric product} of $K$ on level $n$ as $$K_n := \set{g \in \St(n): g|_v \in K \text{ for all $v \in \mathcal{L}_n$}}.$$

We say that $G \leq \Aut(T)$ is \textit{regular branch} over a subgroup $K$ if $G$ is level-transitive and $K_1 \leq_f K \leq_f G$. 

\begin{Lemma}[{\cite[Lemma 10]{Sunic2006}}] 
Let $T$ be a $d$-regular rooted tree and $G \leq \Aut(T)$ be a self-similar regular branch group, branching over a subgroup $K$ containing $\St_G(m)$. Then, for all $n \geq m$, $$\St_G(n) = (\St_G(m))_{n-m},$$ where the right-hand side corresponds to the geometric product of $\St_G(m)$.
\label{lemma: St(n+1) = prod St(n)}
\end{Lemma}

Let $D$ be a positive natural number, $T$ a $d$-regular rooted tree and $\mathcal{P}$ a subgroup of $\Aut(T^D)$. The \textit{group of finite type} of \textit{depth} $D$ and set of \textit{patterns} $\mathcal{P}$ is defined as 
\begin{equation}
G_\mathcal{P} := \set{g \in \Aut(T): g|_v^D \in \mathcal{P} \text{ for all $v \in T$}}.
\label{equation: finite type definition}
\end{equation}

Thus, the elements of $G_\mathcal{P}$ must act at every vertex according to a finite family of allowed actions on the $D$-truncated tree. The smallest $D$ necessary to define the group $G_\mathcal{P}$ is called the \textit{depth}.

\begin{Proposition}[{\cite[page 3]{Bondarenko2014} and \cite[Proposition 7.5]{Grigorchuk2005}}]
If $G$ is a group of finite type of depth $D$ then $G$ is closed in $\Aut(T)$, self-similar and regular branch over $\St_G(D-1)$.
\label{proposition: finite type is closed self-similar and regular branch}
\end{Proposition}

The following result characterizes groups of finite type and their importance in the theory of groups acting on trees:

\begin{theorem}[{\cite[Theorem 3]{Sunic2010}}]
Groups of finite type are the closure of regular branch groups. Furthermore, if $G$ is a regular branch group branching over a subgroup $K$ containing $\St_G(D-1)$, then $\overline{G}$ is a group of finite type of depth $D$.
\label{theorem: characterization finite type groups}
\end{theorem}

If for example $D = 1$, we are forcing the labels of each element in $G_\mathcal{P}$ to be in a certain subgroup $\mathcal{P}$ of $\Sym(d)$. In this case the group is isomorphic to the inverse limit of the \textit{wreath products} $\mathcal{P} \wr \mathcal{P} \wr \dots \wr \mathcal{P}$. We will denote these group by $W_\mathcal{P}$ and call it the \textit{iterated wreath product}. In the case of iterated wreath products, it is straightforward to calculate its Hausdorff dimension: 

\begin{Lemma}
\label{lemma: Hausdorff dimension iterated wreath product}
Let $T$ be a $d$-regular tree and $H \leq \Sym(d)$. Then $$\mathcal{H}(W_H) = \frac{\log(\abs{H})}{\log(d!)}$$
\end{Lemma}

\begin{proof}
Since at any vertex of the truncated tree $T^{n-1}$ we are allowed to put as label any element of $H$, we have $$\abs{\pi_n(W_H)} = \abs{\mathcal{H}}^{\sum_{i = 0}^{n-1} \# \mathcal{L}_i} = \abs{H}^{\frac{d^n-1}{d-1}},$$ since $\# \mathcal{L}_i = d^i$. This in particular applies when $H = \Sym(d)$ that corresponds to the whole groups of automorphisms as $W_{\Sym(d)} = \Aut(T)$. Therefore, 
\begin{align*}
\mathcal{H}(W_H) = \liminf_{n \rightarrow +\infty} \frac{\log(\abs{\pi_n(W_H)})}{\log(\abs{\pi_n(\Aut(T))})} = \frac{\log(\abs{H})}{\log(\abs{\Sym(d)})} = \frac{\log(\abs{H})}{\log(d!)}.
\end{align*}
\end{proof}

\subsection{Fixed-point proportion}

We define the fixed-point proportion of a subgroup $G$ of $\Aut(T)$ acting on a spherically homogeneous tree $T$ as 
\begin{equation}
\FPP(G) = \lim_{n \rightarrow \infty} \frac{\# \set{g \in \pi_n(G): \text{$g$ fixes a vertex in $\mathcal{L}_n$}}}{\abs{\pi_n(G)}}.
\label{equation: FPP definition}
\end{equation}

\begin{Lemma}
The fixed-point proportion of $G$ is well-defined, namely, the limit in \cref{equation: FPP definition} always exists. 
\label{lemma: FPP exists}
\end{Lemma}

\begin{proof}
Consider the surjective map $\pi: \pi_{n+1}(G) \rightarrow \pi_n(G)$, namely, the projection of the action on level $n+1$ to level $n$. Notice that if $h \in \pi_n(G)$ does not fix any vertex at level $n$, then all the preimages of $h$ under $\pi$ will not fix any vertex at level $n+1$. Since $\pi$ is a group homomorphism, the number of preimages is $\abs{\pi_{n+1}(G)}/\abs{\pi_n(G)}$ for any point $h \in \pi_n(G)$ and consequently we have the inequality $$\abs{\pi_{n+1}(G)} - f_{n+1} \geq \frac{\abs{\pi_{n+1}(G)}}{\abs{\pi_n(G)}} \left(\abs{\pi_n(G)} - f_n \right),$$ where $$f_n = \# \set{g \in \pi_n(G): \text{$g$ fixes a vertex in $\mathcal{L}_n$}}.$$ Rewriting the inequality, we obtain $$\frac{f_n}{\abs{\pi_n(G)}} \geq \frac{f_{n+1}}{\abs{\pi_{n+1}(G)}},$$ that is exactly the sequence considered in the limit of the fixed-point proportion. Since the sequence is bounded below by zero and decreasing, the limit must exist.
\end{proof}

\begin{Lemma}
Let $G \leq \Aut(T)$ and $H \leq_f G$. Then
\begin{enumerate}
    \item $\FPP(G) = \FPP(\overline{G})$,
    \item $\FPP(H) \leq [G:H] \, \FPP(G)$. In particular, if $\FPP(G) = 0$ then $\FPP(H) = 0$.
\end{enumerate}
\label{lemma: FPP first properties}
\end{Lemma}

\begin{proof}
For (1), we just observe that $\pi_n(G) = \pi_n(\overline{G})$ for all $n \in \NN$ and the definition of the fixed-point proportion only depends on $\pi_n(G)$. 

For (2), we know that if the index is finite, then there exists $n_0 \in \NN$ such that $[\pi_n(G): \pi_n(H)] = [G:H]$ for all $n \geq n_0$. On the other side, if $K \leq \Aut(T)$ and we let $f_n(K)$ denote the number of elements in $\pi_n(K)$ that fix a vertex in $\mathcal{L}_n$, we have that $f_n(H) \leq f_n(G)$ since $H \leq G$. Therefore 
\begin{align*}
\FPP(H) = \lim_{n \rightarrow +\infty} \frac{f_n(H)}{\abs{\pi_n(H)}} = [G:H] \lim_{n \rightarrow +\infty} \frac{f_n(H)}{\abs{\pi_n(G)}} \leq [G:H] \lim_{n \rightarrow +\infty} \frac{f_n(G)}{\abs{\pi_n(G)}} = \\ [G:H] \,  \FPP(G).
\end{align*}
\end{proof}

By \cref{lemma: Hausdorff dimension first properties}, we have that if $H \leq_f G$ and $G$ is infinite, then $\mathcal{H}_G(H) = 1$. Therefore, a natural generalization for \cref{lemma: FPP first properties}.(2) is to consider the case of $H \leq G$ such that $\mathcal{H}_G(H) = 1$. However, we have the following counterexample. Let $G \leq \Aut(T)$ be an infinite group with null fixed-point proportion and take the subgroup $H := \bigcap_{n \in \NN} \st_G(0^n)$, where $0^n$ is the leftmost vertex at level $n$. Then, we have $\pi_n(H) = \pi_n(\st_G(0^n))$ and therefore every element in $\pi_n(H)$ fixes $0^n$, so $\FPP(H) = 1$. By the orbit-stabilizer theorem, we have $[\pi_n(G): \pi_n(H)] \leq \# \mathcal{L}_n$ and consequently 
\begin{align*}
\mathcal{H}_G(H) \geq 1 - \frac{\log(\# \mathcal{L}_n)}{\log(\abs{\pi_n(G)})}
\end{align*}
Then, it is not hard to find a group $G$ such that the previous limit converges to one. For instance, $G = \Aut(T)$. 

Due to \cref{lemma: FPP first properties}.(1), we may assume that $G$ is a closed subgroup of $\Aut(T)$. Since $\Aut(T)$ is a profinite group and $G$ is closed, then $G$ is compact and has a unique probability Haar measure $\mu$.

A remarkable property of the Haar measure is that if $S$ is a measurable set in $\overline{G}$, then \begin{equation} \mu(S) = \lim_{n \to +\infty} \frac{\# \pi_n(S)}{\abs{\pi_n(\overline{G})}}. \label{equation: Haar measure limit}
\end{equation}

Given $G$, a closed subgroup of $\Aut(T)$, we define the \textit{fixed-point process} of $G$ as the sequence of functions $\set{X_n: G \rightarrow \NN}_{n \in \NN}$ such that $$X_n(g) = \#\set{v \in \mathcal{L}_n: g(v) = v}.$$ By the definition of the Haar measure, we have
\begin{align}
\begin{split}
\FPP(G) = \lim_{n \rightarrow \infty} \frac{\# \set{g \in \pi_n(G): \text{$g$ fixes a vertex in $\mathcal{L}_n$}}}{\abs{\pi_n(G)}} = \\ \mu(\set{g \in G: \text{ $g$ fixes one vertex in $\mathcal{L}_n$ for all $n \in \NN$}}) = \\ \mu \left(\bigcap_{n \geq 0} \set{g \in G: X_n(g) > 0} \right).
\end{split}
\label{equation: FPP with Haar meaure}
\end{align}

So the fixed-point process is involved in the calculation of the fixed-point proportion (see \cite{Jones2014} for more details). 

\begin{Lemma}
Let $T$ be a spherically homogeneous tree and $\set{G_n}_{n \in \NN}$ a sequence of subgroups of $\Aut(T)$ such that $G_n \leq G_{n+1}$ and write $G = \bigcup_{n \in \NN} G_n$. Then $\FPP(G_n) \xrightarrow[n \rightarrow +\infty]{} \FPP(G)$.
\label{lemma: continuity FPP}
\end{Lemma}

\begin{proof}
By \cref{lemma: FPP first properties}, we may assume that the subgroups $G_n$'s and $G$ are closed. Let $\mu_n$ be the Haar measure on $G_n$ and consider $\hat{\mu}_n$ the extension of $\mu_n$ to $G$ defined by $$\hat{\mu}_n(A) = \mu_n(A \cap G_n).$$ In particular, $$\FPP(G_n) = \hat{\mu}_n \left(\bigcap_{n \geq 0} \set{g \in G: X_n(g) > 0} \right)$$

Denote by $\mu$ the Haar measure on $G$. By \cite[Exercise 4, Chapter VII]{Katznelson1968}, the sequence $\hat{\mu}_n \xrightarrow[n \rightarrow +\infty]{} \mu$ in the weak-* topology, so 
\begin{align*}
\FPP(G_n) = \hat{\mu}_n \left(\bigcap_{n \geq 0} \set{g \in G: X_n(g) > 0} \right) \xrightarrow[n \rightarrow +\infty]{} \\ \mu \left(\bigcap_{n \geq 0} \set{g \in G: X_n(g) > 0} \right) = \FPP(G).  
\end{align*}
\end{proof}

\subsection{Iterated Galois groups}
\label{subsection: The galois of iterated rational functions}

Let $K$ be a field, and let $f$ be a rational function of degree at least $2$ (namely, the maximum degree of numerator and denominator is at least $2$) and with coefficients in $K$. We denote the $n$th iteration of $f$ by $f^n$. Let $t$ be a transcendental element over $K$ and fix $K(t)^{sep}$ a separable closure of $K(t)$. Let us denote $f^{-n}(t)$ the set of preimages of $t$ by $f^n$ inside $K(t)^{sep}$ and $K_n(f) := K(f^{-n}(t))$. Then, $K_n(f)/K(t)$ is a Galois extension, and $K_n(f) \subseteq K_{n+1}(f)$ for all $n \geq 0$ (see \cite{Jones2014}). Define $K_\infty(f) := \bigcup_{n \geq 0} K_n(f)$. Then, $K_\infty(f)/K(t)$ is also a Galois extension. For $n \in \NN \cup \set{\infty}$, define also $$G_n(K,f,t) := \Gal(K_n(f)/K(t)).$$ By the definition of $K_\infty(f)$ we have $G_\infty(K,f,t) = \varprojlim G_n(K,f,t)$. So, in particular $G_\infty(K,f,t)$ is a profinite group that acts naturally on $f^{-n}(t)$ by permuting its elements. We call $G_\infty(K,f,t)$ the \textit{iterated Galois group} of $f$ at $t$.

\begin{Lemma}
Let $K$ be a field of characteristic zero, $f$ a rational function in $K(x)$ of degree $d$ and $t$ transcendental over $K$. Then $\# f^{-n}(t) = d^n$.
\label{lemma: t transcendental d^n preimages}
\end{Lemma}

\begin{proof}
Write $f^n(z) = p_n(z)/q_n(z)$ with $p_n,q_n \in K[x]$ such that $\gcd(p_n,q_n) = 1$. Then $z \in f^{-n}(t)$ if and only if we have $p_n(z) - t q_n(z) = 0$. As $\deg(f^n) = d^n$, then $\deg_x(p_n(x) - t q_n(x)) = d^n$, so in order to prove that $\# f^{-n}(t) = d^n$, it is enough to prove that $p_n(x) - t q_n(x)$ is irreducible as a polynomial in $K(t)[x]$ since $K$ is a field of characteristic zero. 

Write $p_n(x) - t q_n(x) = A(x,t) B(x,t)$. As the left hand side is a polynomial in both variables, we may assume $A,B \in K[x,t]$. The left hand side has degree $1$ on $t$, therefore, without loss of generality, we may assume $A \in K[x]$ and $B(x,t) = b_0(x) + tb_1(x)$. Then $Ab_0 = p_n$ and $Ab_1 = q_n$. As $\gcd(p_n, q_n) = 1$, then $A \in K$. 
\end{proof}

In light of \cref{lemma: t transcendental d^n preimages}, we can create a $d$-regular rooted tree $T$ by putting the $n$-th preimages of $t$ by $f$ on the $n$-th level of the tree and connecting a vertex $v \in \mathcal{L}_n$ and a vertex $w \in \mathcal{L}_{n+1}$ if and only if $f(w) = v$. Then, $G_\infty(K,f,t)$ acts on $T$ faithfully and preserving adjacency (see again \cite{Jones2014} for a proof of this fact). Therefore, the iterated Galois group can be embedded as a closed subgroup of $\Aut(T)$ .

\section{The fixed-point proportion of iterated wreath products}
\label{section: FPP of IWP}

We can extend the definition of the fixed-point proportion given in \cref{equation: FPP definition} to any set $X \subseteq \Aut(T)$ in the following way:
\begin{equation}
\FPP(X) := \lim_{n \rightarrow \infty} \frac{\# \set{x\in \pi_n(X): \text{$x$ fixes a vertex in $\mathcal{L}_n$}}}{\abs{\pi_n(X)}}.
\label{equation: definition FPP for sets}
\end{equation}
Contrary to \cref{lemma: FPP exists}, the existence of the limit above is not guaranteed for every set $X$.

Let $S$ be any subset of $\Sym(d)$ and define $$W_S := \set{g \in \Aut(T): g|_v^1 \in S, \forall v \in T}.$$ Notice that if $S = H$ is a subgroup of $\Sym(d)$, the set $W_H$ corresponds to the definition of iterated wreath product given in \cref{equation: finite type definition}. For this reason, we will also refer to $W_S$ as iterated wreath products, even though they are not necessarily subgroups of $\Aut(T)$. We will devote this section to proving that $\FPP(W_S)$ always exists and give a way to calculate it for any set $S$. 

Let $f_n$ denote the number of elements in $\pi_n(W_S)$ that fix at least one vertex in $\mathcal{L}_n$, write $\sigma_n = \abs{\pi_n(W_S)}$ and $p_n$ be the proportion of them, namely, $p_n = \frac{f_n}{\sigma_n}$. The idea will be to give a recurrence formula for $p_n$ that allows us to calculate the fixed-point proportion of $W_S$ as $$\FPP(W_S) = \lim_{n \rightarrow + \infty} p_n.$$  By definition of $W_S$, we have $\sigma_{n+1} = \sigma_n^d \,\, (\# S)$, because at each vertex of level $1$ we can put any element of $\pi_n(W_S)$ and then at the root we have all the options of $S$. Another way to think of $\sigma_{n+1}$ is that we have $\sigma_n$ options for the first $n$ levels and then at each vertex of level $n$ we can put as label any element of $S$, so $\sigma_{n+1} = \sigma_n \,\, (\#S)^{\# \mathcal{L}_n} = \sigma_n \,\, (\# S)^{d^n}$.

\begin{Lemma}
Given $S \subseteq \Sym(d)$, then $\FPP(W_S)$ always exists.
\end{Lemma}

\begin{proof}
To prove that $\FPP(W_S)$ exists, we follow the argument given in \cref{lemma: FPP exists}. Notice that if an element in $\pi_n(W_S)$ does not fix any vertex in level $n$, then any extension to level $n+1$ of that element will not fix any vertex in level $n+1$ either. The number of extensions of each element is $(\#S)^{\# \mathcal{L}_n} = (\#S)^{d^n}$, therefore $$\sigma_{n+1} - f_{n+1} \geq (\#S)^{d^n} (\sigma_n - f_n) = \frac{\sigma_{n+1}}{\sigma_n}(\sigma_n - f_n) = \sigma_{n+1}(1-p_n).$$ Reordering terms, we get $p_{n+1} \leq p_n$. Thus, $\set{p_n}_{n \in \NN}$ is a decreasing sequence bounded by zero and therefore $\FPP(W_S)$ exists.
\end{proof}

\begin{Definition}
Let $d \in \NN$ and $S \subseteq \Sym(d)$. Define $D_S(k)$ as the number of permutations in $S$ having exactly $k$ fixed points.
\label{definition: derangements}
\end{Definition}

This definition generalizes the concept of the number of derangements for any set (see \cite{Hassani2003}). Clearly, $$\sum_{k = 0}^d D_S(k) = \# S.$$

Our next step will be to find a recurrence formula for $f_{n+1}$. Let $s \in S$ be a label for the root with $k$ fixed points. If $k = 0$, the resulting elements will not produce fixed points at level $n+1$, so we assume that $k > 0$. 

An element $g \in \pi_{n+1}(W_S)$ whose label at the root is $s$ will have fixed points at level $n+1$ only in the subtrees fixed by $s$. In the remaining $d-k$ subtrees, any element of $\pi_n(W_S)$ can be placed freely, contributing a factor of $\sigma_n^{d-k}$. 

In the $k$ fixed subtrees, we need to place at least one element of $\pi_n(W_S)$ that fixes a vertex, while the rest may be arbitrary. To avoid repetition, we consider the disjoint sets $$F_n = \set{h \in \pi_n(W_S): \text{$h$ has fixed points on $\mathcal{L}_n$}} \text{ and } F_n^c.$$ We then sum in all the possible $k$-fold Cartesian products of $F_n$ and $F_n^c$, excluding the product $F_n^c \times \dots \times F_n^c$, which does not produce elements with fixed points at level $n+1$. 

To express this formally, let $\beta \in \set{\pm 1}$ and define 
\begin{align*}
\overline{f_n}^\beta = \left \{ \begin{matrix} f_n & \mbox{if $\beta = 1$,} \\ 
\sigma_n - f_n & \mbox{if $\beta = -1$.}
\end{matrix}\right.
\end{align*}

We conclude the following formula: 
\begin{equation}
f_{n+1} = \sum_{k = 1}^d D_S(k) \sigma_n^{d-k} \sum_{(\beta_1,\dots,\beta_k) \in \set{\pm 1}^k \setminus \set{(-1,\dots,-1)}} \prod_{i = 1}^k \overline{f_n}^{\beta_i}
\label{equation: fn+1 first deduction}
\end{equation}

The following lemma will help to clarify the intricate formula:

\begin{Lemma}
\label{lemma: formula fn+1}
Let $d \in \NN$ and $S \subseteq \Sym(d)$. Using the notation introduced earlier, we have
\begin{equation}
f_{n+1} = \sum_{k = 1}^d D_S(k) \sigma_n^{d-k} (\sigma_n^k - (\sigma_n - f_n)^k).
\label{equation: fn+1 second deduction}
\end{equation}
\end{Lemma}

\begin{proof}
First, consider the sum $$\sum_{(\beta_1,\dots,\beta_k) \in \set{\pm 1}^k} \prod_{i = 1}^k \overline{f_n}^{\beta_i}.$$ Fix $t \in \set{0,\dots,k}$. Notice that the number of terms of the form $f_n^t (\sigma_n - f_n)^{k-t}$ corresponds to the number of vectors $(\beta_1,\dots,\beta_k)$ having $t$ entries equal to $1$. To count the number of such vectors, we can think of $\set{1,\dots,k}$ as the set of indices, and we are selecting $t$ of them to assign the value $1$. There are exactly $\binom{k}{t}$ ways to do this. Hence, $$\sum_{(\beta_1,\dots,\beta_k) \in \set{\pm 1}^k} \prod_{i = 1}^k \overline{f_n}^{\beta_i} = \sum_{t = 0}^k \binom{k}{t} f_n^t(\sigma_n - f_n)^{k-t}.$$

By the binomial theorem, $$\sum_{t = 0}^k \binom{k}{t} f_n^t(\sigma_n - f_n)^{k-t} = (f_n + (\sigma_n - f_n))^k = \sigma_n^k.$$

Now, considering the term $(\beta_1,\dots,\beta_k) = (-1,\dots,-1)$ separately, we obtaining: 
$$\sum_{(\beta_1,\dots,\beta_k) \in \set{\pm 1}^k \setminus \set{(-1,\dots,-1)}} \prod_{i = 1}^k \overline{f_n}^\beta_i = \sigma_n^k - (\sigma_n - f_n)^k.$$

Substituting this result into \cref{equation: fn+1 first deduction},
$$f_{n+1} = \sum_{k = 1}^d D_S(k) \sigma_n^{d-k} (\sigma_n^k - (\sigma_n - f_n)^k)$$

and the result follows.
\end{proof}

Dividing \cref{equation: fn+1 second deduction} by $\sigma_{n+1}$, using its recurrence formula $\sigma_{n+1} = \sigma_n^d \,\, \# S$ and replacing $\frac{f_n}{\sigma_n}$ with $p_n$, we obtain

\begin{equation}
p_{n+1} = \frac{f_{n+1}}{\sigma_{n+1}} = \sum_{k = 1}^d \frac{D_S(k)}{\# S} \left( 1 - \left( 1 - p_n \right)^k \right).
\label{equation: pn+1 firt deduction}
\end{equation}

\begin{Definition}
Let $d \in \NN$ and $S \subseteq \Sym(d)$. We define the \textit{characteristic polynomial} of $S$ as $$f_S(x) := \sum_{k = 1}^d \frac{D_S(k)}{\# S} \left( 1 - \left( 1 - x \right)^k \right).$$
\label{definition: characteristic polynomial}
\end{Definition}

\begin{Remark}
Using \cref{definition: characteristic polynomial} and \cref{equation: pn+1 firt deduction} we conclude that $$p_{n+1} = f_S(p_n).$$ Taking limit when $n \rightarrow +\infty$, we deduce that $\FPP(W_S)$ is a fixed point of the characteristic polynomial of $f_S$. 
\label{Remark: FPP is fixed point fS}
\end{Remark}

\begin{Remark}
In \cite{Juul2018}, Juul considered the indicatrix function $\Phi_\Gamma$ to prove \cite[Theorem 1.2]{Juul2018}, where $\Gamma$ is any set of $\Sym(d)$. The functions $\Phi_\Gamma$ and $f_S$ are in fact related by the formula $$f_S(x) = 1 - \Phi_S(1-x),$$ which justifies the similar properties that they have.
\end{Remark}

\begin{Lemma}
Let $d \in \NN$ and $S \subseteq \Sym(d)$. The characteristic polynomial $f_S$ depends only on the conjugacy class of $S$ in $\Sym(d)$.
\label{lemma: characteristic polynomial conjugacy class}
\end{Lemma}

\begin{proof}
Let $S'$ be a set in $\Sym(d)$ such that $S' = gSg^{-1}$ for some $g \in \Sym(d)$. For $x \in \set{1\dots,d}$ and $s \in S$, observe that $x$ is fixed by $gsg^{-1}$ if and only if $g^{-1}(x)$ is fixed by $s$. Since $g$ is a permutation, this implies that $D_{S'}(k) = D_S(k)$ for all $k = 0,\dots,d$. Consequently, $f_{S'} = f_S$.
\end{proof}

\begin{Proposition}
Let $d \in \NN$, $S \subseteq \Sym(d)$, and $f_S: [0,1] \rightarrow \RR$ be the characteristic polynomial of $S$. Then $f_S$ has the following properties:
\begin{enumerate}
\item $f_S(x) \in [0,1]$ for all $x \in [0,1]$ and $f_S(0) = 0$. Moreover, if $\deg(f_S) > 0$, then $f_S(x) = 0$ if and only if $x = 0$.
\item $f_S'(x) \geq 0$, and if $\deg(f_S) > 0$, then $f_S'(x) > 0$ for $x \in [0,1)$.
\item $f_S''(x) \leq 0$ for $x \in [0,1]$, and strictly negative in $[0,1)$ if $\deg(f_S) > 1$.
\item If $f_S$ is not the identity function, the equation $f_S(x) = x$ has at most two solutions in $[0,1]$, with $x = 0$ being one of them.
\end{enumerate}
\label{proposition: properties fS}
\end{Proposition}

\begin{proof}
For (1), since $x \in [0,1]$, we have $0 \leq 1 - (1-x)^k \leq 1$, and it equals zero if and only if $ x= 0$ for all $k$. Furthermore, $D_S(k) \geq 0$ for all $k$, so $$0 \leq f_S(x) \leq \sum_{k = 1}^d \frac{D_S(k)}{\# S} \leq \sum_{k = 0}^d \frac{D_S(k)}{\# S} = 1.$$ The only exception occurs when $\deg(f_S) < 1$, where $f_S(x) = 0$.

\bigskip

For (2), the derivative of $f_S$ is $$f_S'(x) = \sum_{k = 1}^d \frac{D_S(k) k}{\# S} (1-x)^{k-1} \geq 0$$ for $x \in [0,1]$. Moreover, if $\deg(f_S) > 0$, then $(1-x)^{k-1} > 0$ for $x \in [0,1)$, making $f_S'(x)$ is strictly positive in $[0,1)$. 

\bigskip

For (3), the second derivative of $f_S$ is: $$f_S''(x) = - \sum_{k = 2}^d \frac{D_S(k)}{\# S} k(k-1) (1-x)^{k-2} \leq 0$$ for $x \in [0,1]$. If $\deg(f_S) > 1$, then $D_S(k) > 0$ for some $k \geq 2$, and $f_S''(x) < 0$ for $x \in [0,1)$.

\bigskip 

For (4), we use the strict concavity of $f_S$ established in (3). Concaveness means that for all $x,y$ in the domain and $\alpha \in [0,1]$, we have $$f_S((1-\alpha)x + \alpha y ) > (1-\alpha) f_S(x) + \alpha f_S(y).$$ Suppose that $f_S$ has another fixed-point different to $0$ in $[0,1]$ and call $x_1$ the smallest possible. If $x_1 = 1$ we are done, if not, let $y \in (x_1,1]$, $x = 0$ and take $\alpha \in (0,1)$ such that $(1-\alpha)x + \alpha y = \alpha y = x_1$. Then, $$f_S(x_1) = x_1 > \alpha f_S(y).$$ Dividing by $\alpha$, we obtain that $y > f_S(y)$ and this for all $y > x_1$, so we at most have two fixed points.
\end{proof}

We are now ready to conclude the main result about the fixed-point proportion of $W_S$:

\begin{theorem}
Let $d \in \NN$ and $S \subseteq \Sym(d)$. Then: 
\begin{enumerate}
    \item $\FPP(W_S)$ equals the largest fixed point of $f_S$ in $[0,1]$.
    \item $\FPP(W_S)$ is the solution of a polynomial of degree at most $d-1$ with coefficients in $\ZZ\left[\frac{1}{\# S}\right]$. 
    \item $\FPP(W_S) = 0$ if and only if $f_S'(0) \leq 1$ and $f_S$ is not the identity function.
    \item $\FPP(W_S) = 1$ if and only if every element in $S$ fixes at least one element in $\set{1,\dots,d}$.
\end{enumerate}
\label{theorem: FPP(WS)}
\end{theorem}

\begin{proof}
For (1), as it was observed in \cref{Remark: FPP is fixed point fS}, the fixed-point proportion of $W_S$ is a fixed point of $f_S$ in $[0,1]$. 

If $\deg(f_S) < 1$, then $f_S(x) = 0$ and the only fixed point is zero. 

If $\deg(f_S) = 1$, then $f_S'(0) = \frac{D_S(1)}{\# S} \leq 1$. If $f_S'(0) < 1$, then the only fixed point is zero. If $f_S'(0) = 1$, this implies that $D_S(1) = \# S$ or equivalently that every element in $S$ fixes exactly one point. Therefore $\FPP(W_S) = 1$, coinciding with the largest fixed point of $f_S$.

If $\deg(f_S) > 1$, then by \cref{proposition: properties fS}, the function $f_S$ has at most two fixed points, with $0$ as one of them. If $0$ is the only fixed point we are done. Otherwise, let $x_0 > 0$ be the other fixed point. Consider $g(x) = f_S(x)-x$. Then $g(0) = g(x_0) = 0$ and by Rolle's theorem there exists $c \in (0,x_0)$ such that $g'(c) = 0$ or equivalently that $f_S'(c) = 1$. Since $f_S''$ is strictly negative, then $f_S'$ is strictly decreasing but since $f_S'$ is always positive then $0 < f_S'(x_0) < f_S'(c) = 1 < f_S'(0)$. This proves that $0$ is a repelling fixed point whereas $x_0$ is attracting, so $\FPP(W_S) = x_0$. \\

For (2), since $\deg(f_S) \leq d$ and $f_S(0) = 0$, if $\deg(f_S) > 1$, $\FPP(W_S)$ satisfies the polynomial: $$\frac{f_S(x)}{x} - 1 = 0,$$ which has degree at most $d-1$ and coefficients in $\ZZ\left[\frac{1}{\# S}\right]$, as $D_S(k) \in \ZZ$. \\

For (3), the direct was already proved in (1). For the converse, if $f_S'(0) \leq 1$ and $\deg(f_S) > 1$, since $f_S'$ is strictly decreasing, then $f_S(x) < x$ for all $x > 0$ and the only fixed point is $0$. \\

For (4), $\FPP(W_S) = 1$ if and only if $$f_S(1) = \sum_{k = 1}^d \frac{D_S(k)}{\# S} = 1.$$ Since $\#S = \sum_{k = 0}^d D_S(k) = 1$, this implies that $\FPP(W_S) = 1$ if and only if $D_S(0) = 0$, meaning every element in $S$ fixes at least one element in $\set{1,\dots,d}$. 
\end{proof}

A particular case that we will be of interest in this article is when $S$ is a coset in $\Sym(d)$. We first need the following definitions:

\begin{Definition}
Let $G$ be a group acting on a set $Y$ via the action $\rho: G \curvearrowright Y$. Let $S$ be a set of $G$, an element $s \in S$ and $y \in Y$. We denote by $Y^s := \set{y \in Y: sy = y}$ and by $\Stab_S(y) := \set{s \in S: sy = y}$. 
\end{Definition}

\begin{Remark}
If we consider $G = \Sym(d)$ with the natural action on $X = \set{1,\dots,d}$ and a set $S \subseteq \Sym(d)$, we have 
\begin{align*}
\frac{1}{\# S} \sum_{s \in S} \# X^s = \frac{1}{\# S} \sum_{k = 0}^d \,\, \sum_{s \in S: \# X^s = k} k = \frac{1}{\# S} 
\sum_{k = 0}^d D_S(k) k = \frac{1}{\# S} 
\sum_{k = 1}^d D_S(k) k = f_S'(0).
\end{align*}
\label{Remark: Burnside and derivative fS}
\end{Remark}

If $S$ is a coset, the quantity on the left-hand side can be studied as a generalization of Burnside's lemma:

\begin{Lemma}[{Burnside Lemma for cosets}]
Let $G \curvearrowright Y$ be an action of a finite group $G$ over a finite set $Y$. Let $H \leq G$ and $A = gH$ be a coset of $H$ in $G$. 

\begin{enumerate}
    \item If $H \curvearrowright Y$ acts transitively, then $$\frac{1}{\#A} \sum_{a \in A} \# Y^a = 1.$$
    \item Define $Y^* := \set{y \in Y: \exists h \in H \text{ s.t. } hy = g^{-1}y}$. If $g \in N_G(H)$, the normalizer of $H$ in $G$, then $$\frac{1}{\#A} \sum_{a \in A} \# Y^a = \#(Y^*/H).$$
\end{enumerate}
\label{lemma: Burnside cosets}
\end{Lemma}

\begin{proof}
Consider the set $Z = \set{(a,y) \in A \times Y: ay = y}$. We can count the elements of $Z$ in two different ways. On one hand, we have: $$\#Z = \sum_{a \in A} \# Y^a.$$

On the other hand, we have: $$\# Z = \sum_{y \in Y} \# \Stab_A(y).$$

Notice that $a \in \Stab_A(y)$ if and only if $ay = y$. Since $A = gH$, there exists $h \in H$ such that $a = gh$, so $ay = y$ if and only if $hy = g^{-1}y$. If no such $h$ exists, then $\Stab_A(y) = \emptyset$, so we can restrict the sum to $Y^*$. If $y \in Y^*$, and $h_0 \in H$ is an element such that $h_0 y= g^{-1}y$, then we have a bijection between $\Stab_H(y)$ and $\Stab_A(y)$ given by $h \mapsto gh_0h$. Therefore $$\#Z = \sum_{y \in Y^*} \# \Stab_H(y).$$ By the orbit-stabilizer theorem, we know that $$\# \Stab_H(y) = \frac{\abs{H}}{\# \orb_H(y)},$$ so $$\#Z = \abs{H} \sum_{y \in Y^*} \frac{1}{\# \orb_H(y)}.$$

If $H$ is transitive, then $Y^* = Y$, and thus $$\# Z = \abs{H} \sum_{B \in Y/H} \sum_{y \in B} \frac{1}{\# B} = \abs{H} \# (Y/H) = \abs{H} = \#A.$$ This gives the formula in the first case. 

If $g \in N_G(H)$, the action $G \curvearrowright Y$ restricts to $H \curvearrowright Y^*$. Indeed, if $y \in Y^*$, there exists $h_0 \in H$ such that $h_0y = g^{-1}y$. Then, for $h \in H$, there exists $h' \in H$ such that $g^{-1}h = h' g^{-1}$. Therefore $$g^{-1}(hy) = h' g^{-1}y = h'h_0 y = (h'h_0 h^{-1}) h(y).$$ Since $h'h_0h^{-1} \in H$, we obtain that $hy \in Y^*$. 

Thus, we can split the sum over the orbits $Y^*/H$, obtaining: $$\# Z = \abs{H} \sum_{B \in Y^*/H} \sum_{y \in B} \frac{1}{\# B} = \abs{H} \# (Y^*/H) = \#A \,\, \# (Y^*/H).$$ This gives the formula in the second case.
\end{proof}

In the case that $S$ is a subgroup in \cref{theorem: FPP(WS)}, we deduce the following corollary for the fixed-point proportion of iterated wreath products:

\begin{corollary}
Let $d \in \NN$ with $d \geq 2$, the set $X = \set{1,\dots,d}$ and $\mathcal{P} \leq \Sym(d)$. Then the fixed point proportion of $W_\mathcal{P}$ is: 
\begin{align*}
\text{FPP}(W_\mathcal{P}) = \left \{ \begin{matrix} 
0 & \mbox{if $\mathcal{P}$ is transitive over $X$,} \\ 
\alpha \in (0,1) & \mbox{if $\mathcal{P}$ not transitive and $\exists \, \sigma$ with no fixed points} \\
1 & \mbox{if $\mathcal{P}$ not transitive and every element fixes a point.}
\end{matrix}\right.
\end{align*}
Moreover, the value $\FPP(W_\mathcal{P})$ is the solution of a polynomial of degree $d-1$ with coefficients in $\ZZ\left[\frac{1}{\abs{\mathcal{P}}}\right]$.
\label{corollary: FPP iterated wreath products}
\end{corollary}

\begin{proof}
Since $\mathcal{P}$ is a subgroup, the identity is in $\mathcal{P}$ so $D_\mathcal{P}(d) = 1$ and therefore $\deg(f_\mathcal{P}) = d \geq 2$. The case when $\FPP(W_\mathcal{P}) = 1$, follows directly from \cref{theorem: FPP(WS)}. The case when $\FPP(\mathcal{P}) = \alpha \in (0,1)$ follows from the fact that $\alpha$ is the largest fixed point of $f_\mathcal{P}$ and $f_\mathcal{P}(0) = 0$, so it is a solution of $\frac{f_\mathcal{P}(x)}{x} - 1 = 0$ that is a polynomial with degree $d-1$. 

Finally, by \cref{theorem: FPP(WS)}, we have $\FPP(W_\mathcal{P}) = 0$ if and only if $f_\mathcal{P}'(0) \leq 1$ and $f_\mathcal{P}$ is not the identity function. Since $\deg(f_\mathcal{P}) > 1$, then $\FPP(W_\mathcal{P}) = 0$ if and only if $f_\mathcal{P}'(0) \leq 1$. By \cref{Remark: Burnside and derivative fS} and \cref{lemma: Burnside cosets}, we have $$f_\mathcal{P}'(0) = \frac{1}{\abs{\mathcal{P}}} \sum_{k = 1}^d D_\mathcal{P}(k) k = \#(X/\mathcal{P}),$$ so $\FPP(W_\mathcal{P}) = 0$ if and only if $\mathcal{P}$ is transitive over $X$. 
\end{proof}

Running a program in SAGE \cite{sagemath}, we compute the characteristic polynomials of the subgroups of $\Sym(d)$ for $d = 3$ and $4$. By \cref{lemma: characteristic polynomial conjugacy class}, it suffices to consider the conjugacy classes of subgroups in $\Sym(d)$, since they have the same characteristic polynomial.

For $d = 3$, the group $\Sym(3)$ has four conjugacy classes. Representatives of these classes are: $\set{1}$, $\<(2,3)>$, $\<(1,2,3)>$ and $\Sym(3)$. The corresponding fixed-point proportions are $1$, $1$, $0$ and $0$, respectively. This matches with \cref{corollary: FPP iterated wreath products}, as the first two classes fix at least one point and the last two are transitive. \cref{figure: FPPS3} illustrates the functions $f_\mathcal{P}$ for each conjugacy class of $\mathcal{P}$ in $\Sym(3)$. Notice that the intersections of $f_\mathcal{P}$ with the identity function are only at $0$ and $1$, as expected. In particular, there are no examples of iterated wreath products with non-trivial fixed-point proportion for $d = 3$.  

\begin{figure}[h!]
\begin{tikzpicture}
\begin{axis}
[axis lines = left]
\addplot [ domain=0:1, samples=100, color=Red,] {(x - 1)^3 + 1};
\addplot [domain=0:1, samples=100, color=Green,] { 1/2*(x - 1)^3 + 1/2*x + 1/2};
\addplot [domain=0:1, samples=100, color=Periwinkle,] {1/6*(x - 1)^3 + 1/2*x + 1/6};
\addplot [domain=0:1, samples=100, color=Rhodamine,] {1/3*(x - 1)^3 + 1/3};
\addplot [domain=0:1, samples=100, color=black,] {x};
\addplot[ color=black, mark=*, ] coordinates {(0,0)};
\addplot[ color=black, mark=*, ] coordinates {(1,1)};
\end{axis}
\end{tikzpicture}
\caption{Plot of the functions $f_\mathcal{P}$ for $\mathcal{P}$ subgroups of $\Sym(3)$.}
\label{figure: FPPS3}
\end{figure}
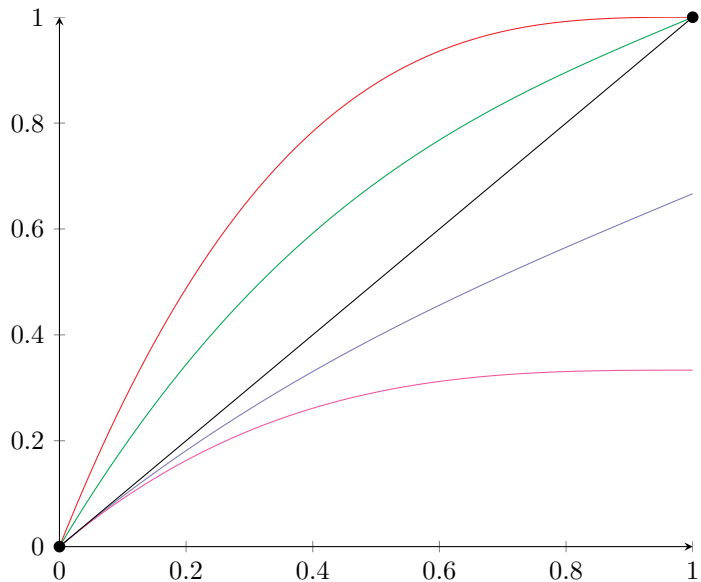

For $d = 4$, the group $\Sym(4)$ has eleven conjugacy classes. In four of them (the ones corresponding to $\Sym(3)$) their elements always fix a point, so their fixed-point proportion is $1$. There are five transitive conjugacy classes, namely, their fixed-point proportion is $0$. In addition, there are two conjugacy classes whose fixed-point proportion is not trivial. Possible representatives for these are $\<(1,2),(3,4)>$ and $\<(3,4),(1,2)(3,4)>$. Notice that these classes are not transitive because their elements cannot send $1$ to $3$, nor do they fix a particular point. The corresponding fixed-point proportions are approximately $0.45631\dots$ and $0.70440\dots$ respectively. \cref{figure: FPPS4} shows the functions $f_\mathcal{P}$ for each conjugacy class of $\mathcal{P}$ in $\Sym(4)$.

\begin{figure}[h!]
\begin{tikzpicture}
\begin{axis}
[axis lines = left]
\addplot [ domain=0:1, samples=100, color=Red,] {-(x-1)^4+1};
\addplot [domain=0:1, samples=100, color= Orange,] {1/2*(-(x-1)^4-(x-1)^2+2)};
\addplot [domain=0:1, samples=100, color=Dandelion,] {1/6*(-(x-1)^4-3*(x-1)^2+2*x+4)};
\addplot [domain=0:1, samples=100, color=Green,] {1/3*(-(x-1)^4+2*x+1)};
\addplot [domain=0:1, samples=100, color=PineGreen,] {1/4*(-(x-1)^4-2*(x-1)^2+3)};
\addplot [domain=0:1, samples=100, color=Aquamarine,] {1/2*(-(x-1)^4+1)};
\addplot [domain=0:1, samples=100, color=Periwinkle,] {-1/12*(x - 1)^4 + 2/3*x + 1/12};
\addplot [domain=0:1, samples=100, color=Blue,] {-1/24*(x - 1)^4 - 1/4*(x - 1)^2 + 1/3*x + 7/24};
\addplot [domain=0:1, samples=100, color=Purple,] {-1/8*(x - 1)^4 - 1/4*(x - 1)^2 + 3/8};
\addplot [domain=0:1, samples=100, color=Rhodamine,] {1/4*(-(x-1)^4+1)};
\addplot [domain=0:1, samples=100, color=black,] {x};
\addplot[ color=black, mark=*, ] coordinates {(0,0)};
\addplot[ color=black, mark=*, ] coordinates {(1,1)};
\addplot[ color=black, mark=*, ] coordinates {(0.4563109873079255,0.4563109873079255)};
\addplot[ color=black, mark=*, ] coordinates {(0.7044022574778126,0.7044022574778126)};
\end{axis}
\end{tikzpicture}
\caption{Plot of the functions $f_\mathcal{P}$ for $\mathcal{P}$ subgroups of $\Sym(4)$.}
\label{figure: FPPS4}
\end{figure}
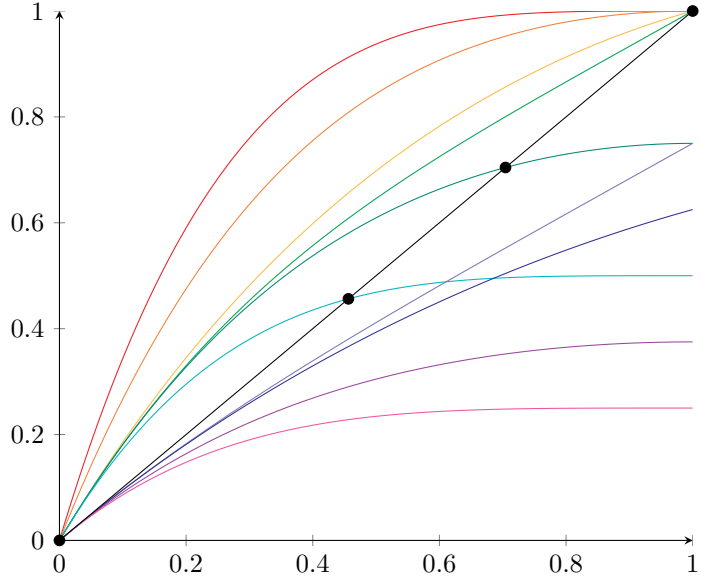

\section{The main construction}
\label{section: the main construction}

In this section, we will use \cref{theorem: FPP(WS)} to construct a family of regular branch groups whose fixed-point proportion and Hausdorff dimension can be computed explictly.

Given $d \geq 2$ and $T$ a $d$-regular tree, consider two subgroups $\mathcal{Q}, \mathcal{P} \leq \Sym(d)$ such that $1 \neq \mathcal{Q} \lhd \mathcal{P}$. Define $$G_\mathcal{Q}^\mathcal{P} = \set{g \in \Aut(T): g|_v^1 \in \mathcal{P}, \, (g|_v^1)(g|_w^1)^{-1} \in \mathcal{Q} \text{ for all $v,w \in T$}}.$$ 

In other words, the elements in this group lie in the iterated wreath product $W_\mathcal{P}$ (see \cref{section: preliminaries} for the definition of $W_\mathcal{P}$) and each element has its labels in the same coset of $\mathcal{P}/\mathcal{Q}$. For example, if $\mathcal{Q} = \mathcal{P}$, then $G_\mathcal{Q}^\mathcal{P} = W_\mathcal{P}$.

\begin{Lemma}
Let $T$ be a $d$-regular tree with $d \geq 2$ and $1 \neq \mathcal{Q} \lhd \mathcal{P} \leq \Sym(d)$. Then, the group $G_\mathcal{Q}^\mathcal{P}$ satisfies the following properties:
\begin{enumerate}
\item $G_\mathcal{Q}^\mathcal{P}$ contains $W_\mathcal{Q}$,
\item $W_\mathcal{Q} \lhd G_\mathcal{Q}^\mathcal{P}$,
\item $[G_\mathcal{Q}^\mathcal{P} : W_\mathcal{Q}] = [\mathcal{P}: \mathcal{Q}]$.
\end{enumerate}
\label{lemma: GQP group and index}
\end{Lemma}

\begin{proof}
We start proving that $G_\mathcal{Q}^\mathcal{P}$ is a group. Clearly, the identity element is in $G_\mathcal{Q}^\mathcal{P}$. Next, suppose $g,h \in G_\mathcal{Q}^\mathcal{P}$. We must verify that $gh^{-1} \in G_\mathcal{Q}^\mathcal{P}$. For $v,w \in T$, by the properties of sections (see \cref{equation: properties sections}), we have:
\begin{align*}
(gh^{-1})|_v^1 \, \, ((gh^{-1})|_w^1))^{-1} = (g|_{h^{-1}(v)}^1) \, (h|_{h^{-1}(v)}^1)^{-1} \, (h|_{h^{-1}(w)}^1) \,  (g|_{h^{-1}(w)}^1)^{-1}.
\end{align*}

Rewriting, this becomes:

\begin{align*}
\left[ (g|_{h^{-1}(v)}^1) \, (h|_{h^{-1}(v)}^1)^{-1} \, (h|_{h^{-1}(w)}^1) \, (g|_{h^{-1}(v)}^1)^{-1} \right]  \left[ ((g|_{h^{-1}(v)}^1)  (g|_{h^{-1}(w)}^1)^{-1}) \right].
\end{align*}

The second term between brackets is in $\mathcal{Q}$ since $g \in G_\mathcal{Q}^\mathcal{P}$ and the first term in brackets is in $\mathcal{Q}$ because $h \in G_\mathcal{Q}^\mathcal{P}$ and $\mathcal{Q} \lhd \mathcal{P}$. This ensures that $gh^{-1} \in G_\mathcal{Q}^\mathcal{P}$. Hence, $G_\mathcal{Q}^\mathcal{P}$ is a group.

We now prove that $W_\mathcal{Q} \lhd G_\mathcal{Q}^\mathcal{P}$. The proof is similar to the previous calculation. Specifically, we need to show that if $g \in G_\mathcal{Q}^\mathcal{P}$ and $h \in W_\mathcal{Q}$ then $(ghg^{-1})|_v^1 \in \mathcal{Q}$ for all $v \in T$. Indeed,
\begin{align*}
(ghg^{-1})|_v^1 = g|_{hg^{-1}(v)}^1 \, h|_{g^{-1}(v)}^1 \, (g|_{g^{-1}(v)}^1)^{-1}.
\end{align*}

Rewriting, this becomes:

\begin{align*}
 (ghg^{-1})|_v^1 = \left[ g|_{hg^{-1}(v)}^1 \, h|_{g^{-1}(v)}^1 \, (g|_{hg^{-1}(v)}^1)^{-1} \right] \, \left[ (g|_{hg^{-1}(v)}^1) \, (g|_{g^{-1}(v)}^1)^{-1} \right].  
\end{align*}

The first term in brackets belongs to $\mathcal{Q}$ because $\mathcal{Q} \lhd \mathcal{P}$, and the second term in bracket belongs to $\mathcal{Q}$ since $g \in G_\mathcal{Q}^\mathcal{P}$. Thus, $(ghg^{-1})|_v^1 \in \mathcal{Q}$ and consequently $W_\mathcal{Q} \lhd G_\mathcal{Q}^\mathcal{P}$.

Finally, consider the map 
\begin{align}
\begin{split}
G_\mathcal{Q}^\mathcal{P}/W_\mathcal{Q} & \rightarrow \mathcal{P}/\mathcal{Q} \\ 
[g] & \mapsto [g|_\emptyset^1].
\label{equation: cosets of GQP}
\end{split}
\end{align}
The map is well-defined because if $h \in W_\mathcal{Q}$ and $g \in G_\mathcal{Q}^\mathcal{P}$ then $$[(gh)|_\emptyset^1] = [g|_\emptyset^1] [h|_\emptyset^1] = [g|_\emptyset^1],$$ and it is in fact a morphism since the root of the tree is fixed by any element. 

To show surjectivity, let $\sigma \in \mathcal{P}$.  Define $g \in \Aut(T)$ such that $g|_v^1 = \sigma$ for all $v \in T$. Then, $g \in G_\mathcal{Q}^\mathcal{P}$ and $[g|_\emptyset^1] = [\sigma]$. 

To prove injectivity, suppose $[g|_\emptyset^1] = [h|_\emptyset^1]$. Then: $$(g|_\emptyset^1)(h|_\emptyset^1)^{-1} = (gh^{-1})|_\emptyset^1 \in \mathcal{Q}.$$ Since $gh^{-1} \in G_\mathcal{Q}^\mathcal{P}$, then $(gh^{-1})|_v^1 \in \mathcal{Q}$ for all $v \in T$ and consequently $gh^{-1} \in W_\mathcal{Q}$. 
\end{proof}

Notice that if $d = 2$, we can only construct one group, since $\Sym(2)$ only has two subgroups. In this case, $\mathcal{Q} = \mathcal{P} = \Sym(2)$ and $G_\mathcal{Q}^\mathcal{P} = \Aut(T)$, that by \cref{corollary: FPP iterated wreath products} has null fixed-point proportion.

\begin{Proposition}
Let $T$ be a $d$-regular tree with $d \geq 3$ and $1 \neq \mathcal{Q} \lhd \mathcal{P} \leq \Sym(d)$. Then $G_\mathcal{Q}^\mathcal{P}$ is a group of finite type of depth $2$ and Hausdorff dimension $$\mathcal{H}(G) = \frac{\log(\abs{\mathcal{Q}})}{\log(d!)}.$$ In particular, its Hausdorff dimension is always positive.
\label{proposition: GQP finite type Hausdorff dimension}
\end{Proposition}

\begin{proof}
By \cref{section: Groups of finite type}, the group $W_\mathcal{Q}$ is regular branch, branching over $\St_{W_\mathcal{Q}}(1)$. Furthermore, $\St_{G_\mathcal{Q}^\mathcal{P}}(1) = \St_{W_\mathcal{Q}}(1)$, so $$(\St_{W_\mathcal{Q}}(1))_1 \leq_f \St_{W_\mathcal{Q}}(1) \leq_f G_\mathcal{Q}^\mathcal{P}.$$ Therefore, $G_\mathcal{Q}^\mathcal{P}$ is regular branch.

To determine the depth of $G_\mathcal{Q}^\mathcal{P}$, we apply \cref{theorem: characterization finite type groups}. First, $G_\mathcal{Q}^\mathcal{P}$ is not an iterated wreath product, so its depth is greater than $1$. Since $\St_{G_\mathcal{Q}^\mathcal{P}}(1) = \St_{W_\mathcal{Q}}(1)$, it follows by \cref{lemma: St(n+1) = prod St(n)} that the group $G_\mathcal{Q}^\mathcal{P}$ is branching over $\St_{G_\mathcal{Q}^\mathcal{P}}(1)$. Therefore $G_\mathcal{Q}^\mathcal{P}$ is of finite type with depth $D = 2$ by \cref{theorem: characterization finite type groups}, as $D-1 = 1$.

Finally, we compute the Hausdorff dimension of $G_\mathcal{Q}^\mathcal{P}$. By \cref{lemma: Hausdorff dimension first properties}, we have $\mathcal{H}(G_\mathcal{Q}^\mathcal{P}) = \mathcal{H}(W_{\mathcal{Q}})$. Then, by \cref{lemma: Hausdorff dimension iterated wreath product}, the result follows.
\end{proof}

The next question concerns the level-transitivity of $G_\mathcal{Q}^\mathcal{P}$. 

\begin{Proposition}
Let $T$ be a $d$-regular tree with $d \geq 3$ and $1 \neq \mathcal{Q} \lhd \mathcal{P} \leq \Sym(d)$. Then $G_\mathcal{Q}^\mathcal{P}$ is level-transitive if and only if $\mathcal{Q}$ is transitive on $X = \set{1, \dots,d}$. 
\label{proposition: GQP level transitive}
\end{Proposition}

\begin{proof}
For the direct, let $x,y \in X$. Since $G_\mathcal{Q}^\mathcal{P}$ is level-transitive, there exists $g \in G_\mathcal{Q}^\mathcal{P}$ such that $g(xx) = xy$, where $xx$ and $xy$ represent vertices at level $2$ as words in $X$. Denote $a_0 = g|_\emptyset^1$ and $a_1 = g|_x^1$. By the definition of $G_\mathcal{Q}^\mathcal{P}$, there exists $q \in \mathcal{Q}$ such that $a_1 = a_0 q$. Since $\mathcal{Q} \lhd \mathcal{P}$, there exists $q' \in \mathcal{Q}$ such that $a_0 q = q' a_0$. Thus, $a_1 = q' a_0$, or equivalently, $a_1 a_0^{-1} \in \mathcal{Q}$. Furthermore, $$a_1 a_0^{-1}(x) = a_1(x) = y,$$ implying that $\mathcal{Q}$ is transitive on $X$.

For the converse, if $\mathcal{Q}$ is transitive on $X$, then $W_\mathcal{Q}$ is level-transitive and since $W_\mathcal{Q} \leq G_\mathcal{Q}^\mathcal{P}$, then $G_\mathcal{Q}^\mathcal{P}$ is level-transitive as well. 
\end{proof}

A closed group is said to be \textit{topologically finitely generated} if it contains a finitely generated dense subgroup.

\begin{Proposition}
Let $T$ be a $d$-regular tree with $d \geq 2$ and $1 \neq \mathcal{Q} \lhd \mathcal{P} \leq \Sym(d)$.

\begin{enumerate}
\item If $\mathcal{Q} \neq \mathcal{Q}'$ or there exists $x \in X = \set{1,\dots,d}$ such that $q(x) = x$ for all $q \in \mathcal{Q}$, then $G_\mathcal{Q}^\mathcal{P}$ is not topologically finitely generated. 
\item If $\mathcal{Q}$ is transitive, then $G_\mathcal{Q}^\mathcal{P}$ is topologically finitely generated if and only if $\mathcal{Q} = \mathcal{Q}'$.
\end{enumerate}
\label{proposition: GQP tfg}
\end{Proposition}

\begin{proof}
Since $W_\mathcal{Q}$ has finite index in $G_\mathcal{Q}^\mathcal{P}$, the group $G_\mathcal{Q}^\mathcal{P}$ i topologically finitely generated if and only if $W_\mathcal{Q}$ is.

For the first part, consider the map
\begin{align*}
\varphi_1: W_\mathcal{Q} \rightarrow \prod_{n \in \NN} \mathcal{Q}/\mathcal{Q}', \quad g \mapsto \prod_{n \in \NN} \left( \prod_{v \in \mathcal{L}_n} (g|_{v}^1) \mathcal{Q}'  \right), 
\end{align*}
where the outer product is the Cartesian product and the inner product is the product of the sections. The map is well-defined since we are taking the quotient by $\mathcal{Q}'$.

The map is clearly surjective. If $\mathcal{Q} \neq \mathcal{Q}'$, the image is not topologically finitely generated, and thus $W_\mathcal{Q}$ is not topologically finitely generated. 

Suppose there exists $x \in X$ such that $q(x) = x$ for all $q \in \mathcal{Q}$. Consider the map
\begin{align*}
\varphi_2: W_\mathcal{Q} \rightarrow \prod_{n \in \NN} \mathcal{Q}, \quad g \mapsto \prod_{n \in \NN} \left( g|_{x^n}^1 \right), 
\end{align*}
where $x^n$ represents the vertex at level $n$ reached by following the path labeled by $x$ at each level. 

The map is a morphism since $g(x^n) = x^n$ for all $g \in W_\mathcal{Q}$. The map is surjective, and since $\mathcal{Q} \neq 1$, the image is not topologically finitely generated. Hence, $W_\mathcal{Q}$ is not topologically finitely generated.

If $\mathcal{Q}$ is transitive, the result follows directly from \cite[Corollary 3]{RadiFiniteType2025} or \cite[Theorem 1]{Bondarenko2010} 
\end{proof}

\begin{Remark}
A profinite group is just-infinite if for every closed non-trivial normal subgroup $N \lhd G$, then $[G:N] < \infty$. A profinite group is strongly complete if its profinite completion coincides with the group. By a recent result proved in \cite[Theorem 1]{RadiFiniteType2025}, which applies to all the groups of the form $G_\mathcal{Q}^\mathcal{P}$, the following properties are equivalent:

\begin{itemize}
\item $G_\mathcal{Q}^\mathcal{P}$ is topologically finitely generated
\item $G_\mathcal{Q}^\mathcal{P}$ is just-infinite
\item $G_\mathcal{Q}^\mathcal{P}$ is strongly complete
\item $\mathcal{Q} = \mathcal{Q}'$
\end{itemize}

\end{Remark}

We conclude this section by showing that the fixed-point proportion of the group $G_\mathcal{Q}^\mathcal{P}$ can be calculated explicitly in terms of its cosets:

\begin{Proposition}
Let $T$ be a $d$-regular tree with $d \geq 2$ and $1 \neq \mathcal{Q} \lhd \mathcal{P} \leq \Sym(d)$. Then \begin{align}
\FPP(G_\mathcal{Q}^\mathcal{P}) = \frac{1}{[\mathcal{P}: \mathcal{Q}]} \sum_{A \in \mathcal{P}/\mathcal{Q}} \FPP(W_A). 
\end{align}
In particular, $\FPP(G_\mathcal{Q}^\mathcal{P}) > 0$ if and only $\FPP(W_A) > 0$ for one coset $A$.
\label{Proposition: FPP(GQP) in terms of FPP(A)}
\end{Proposition}

\begin{proof}

Using the map in \cref{equation: cosets of GQP}, we deduce that cosets of $G_\mathcal{Q}^\mathcal{P}$ by $W_\mathcal{Q}$ are of the form $$W_A := \set{g \in \Aut(T): g|_v^1 \in A, \forall v \in T},$$ where $A$ is a coset of $\mathcal{P}/\mathcal{Q}$. Therefore, if $A,A'$ are two different cosets of $\mathcal{P}/\mathcal{Q}$, then $\pi_n(W_A) \cap \pi_n(W_{A'}) = \emptyset$ for all $n \geq 1$, and $$\bigcup_{A \in \mathcal{P}/\mathcal{Q}} \pi_n(W_A) = \pi_n(G_\mathcal{Q}^\mathcal{P}).$$ This allows us to split the calculation of $\FPP(G_\mathcal{Q}^\mathcal{P})$ as:
\begin{align*}
\FPP(G_\mathcal{Q}^\mathcal{P}) = \lim_{n \rightarrow +\infty} \frac{\#\set{g \in \pi_n(G_\mathcal{Q}^\mathcal{P}): \text{$g$ fixes a vertex in $\mathcal{L}_n$}}}{\abs{\pi_n(G_\mathcal{Q}^\mathcal{P})}} = \\  \sum_{A \in \mathcal{P}/\mathcal{Q}} \lim_{n \rightarrow +\infty} \frac{1}{[\pi_n(G_\mathcal{Q}^\mathcal{P}):\pi_n(W_\mathcal{Q})]} \, \frac{\#\set{g \in \pi_n(W_A): \text{$g$ fixes a vertex in $\mathcal{L}_n$}}}{\abs{\pi_n(W_\mathcal{Q})}}.
\end{align*}

By \cref{lemma: GQP group and index}, we have $[\mathcal{P}: \mathcal{Q}] = [\pi_n(G_\mathcal{Q}^\mathcal{P}):\pi_n(W_\mathcal{Q})]$ and $\abs{\pi_n(W_\mathcal{Q})} = \abs{\pi_n(W_\mathcal{A})}$ for any coset in $\mathcal{P}/\mathcal{Q}$. Therefore
\begin{align*}
\FPP(G_\mathcal{Q}^\mathcal{P}) = \frac{1}{[\mathcal{P}:\mathcal{Q}]}  \sum_{A \in \mathcal{P}/\mathcal{Q}} \lim_{n \rightarrow +\infty} \frac{\#\set{g \in \pi_n(W_A): \text{$g$ fixes a vertex in $\mathcal{L}_n$}}}{\abs{\pi_n(W_A)}}.
\end{align*}

The last limit corresponds to our extension of the definition of the fixed-point proportion of sets made in \cref{equation: definition FPP for sets}. We conclude that
\begin{align*}
\FPP(G_\mathcal{Q}^\mathcal{P}) = \frac{1}{[\mathcal{P}: \mathcal{Q}]} \sum_{A \in \mathcal{P}/\mathcal{Q}} \FPP(W_A). 
\end{align*}
\end{proof}

\section{Examples}
\label{section: examples}

In this subsection, we calculate the fixed-point proportion for explicit cases of level-transitive groups $G_\mathcal{Q}^\mathcal{P}$. By \cref{Proposition: FPP(GQP) in terms of FPP(A)}, we have  $$\FPP(G_\mathcal{Q}^\mathcal{P}) = \frac{1}{[\mathcal{P}:\mathcal{Q}]} \sum_{A \in \mathcal{P}/\mathcal{Q}} \FPP(W_A),$$ where each $A$ is a coset of $\mathcal{P}/\mathcal{Q}$. 

By \cref{proposition: GQP level transitive}, the subgroup $\mathcal{Q}$ must act transitively on $X = \set{1,\dots,d}$. Since $\mathcal{Q}$ is transitive, by \cref{lemma: Burnside cosets}.(1) applied to $\mathcal{P} \curvearrowright X$, we have $$f_A'(0) = \frac{1}{\# A} \sum_{a \in A} \# X^a = 1.$$

Therefore, by \cref{theorem: FPP(WS)}.(3), we have only two possibilities:
\begin{align*}
\FPP(W_A) = \left \{ \begin{matrix} 1 & \mbox{if every element in $A$ fixes exactly one point in $X$,} \\ 
0 & \mbox{otherwise.}
\end{matrix}\right.
\end{align*}

In particular $\FPP(W_\mathcal{Q}) = 0$ since the identity is in $\mathcal{Q}$.

\subsection{Construction 1}

Seeing $X = \set{1, \dots, d}$ as $\ZZ/d\ZZ$, the integer numbers modulo $d$, we define $$\mathcal{P} = \set{x \mapsto ax+b: a \in I, b \in \ZZ/d\ZZ},$$ where $I \leq (\ZZ/d\ZZ)^\times$ and $$\mathcal{Q} = \set{x \mapsto x+b: b \in \ZZ/d\ZZ}.$$ 

Here, $\mathcal{Q} \lhd \mathcal{P}$, and $\mathcal{Q}$ is transitive on $X$. By \cref{proposition: GQP tfg}, the group $G_\mathcal{Q}^\mathcal{P}$ is not topologically finitely generated.

The cosets of $\mathcal{P}/\mathcal{Q}$ are the sets $$A_a = \set{x \mapsto ax+b: b \in \ZZ/d\ZZ}.$$ To determine how many fixed points an element of the form $ax+b$ has, we use the following lemma:

\begin{Lemma}
Let $n \in \NN_{> 0}$ and $\alpha, \beta \in \ZZ/n\ZZ$. Then, the equation $$\alpha x \equiv \beta \pmod{n}$$ has a solution if and only if $\gcd(\alpha,n) \mid \beta$ and the number of solutions is $\gcd(\alpha, n)$. 
\label{lemma: solutions linear equations modulo n}
\end{Lemma}

Applying \cref{lemma: solutions linear equations modulo n} with $(\alpha, \beta, n) = (a-1, -b, d)$ we conclude  
\begin{align*}
\FPP(W_{A_a}) = \left \{ \begin{matrix} 1 & \mbox{if $\gcd(a-1,d) = 1$,} \\ 
0 & \mbox{$\gcd(a-1,d) > 1$.}
\end{matrix}\right.
\end{align*}

Thus,
\begin{align}
\FPP(G_\mathcal{Q}^\mathcal{P}) = \frac{\#\set{a \in I: a-1 \in (\ZZ/d\ZZ)^\times}}{\Phi(d)},
\label{equation: FPP Galois x^d+1}
\end{align}
where $\Phi(d)$ is the Euler's totient function. 

In the particular case where $I = (\ZZ/d\ZZ)^\times$, we can write \cref{equation: FPP Galois x^d+1} in a more compact form. This case will be of particular interest in \cref{section: Application to iterated Galois groups} since it corresponds to the Galois group $G_\infty(\QQ, x^d+c,t)$.

\begin{Proposition}
Define $\psi: \NN \rightarrow \NN$ the function $$\psi(d) = \# \set{a \in (\ZZ/d\ZZ)^\times: a-1 \in (\ZZ/d\ZZ)^\times}.$$ Then $$\psi(d) = d \prod_{p \mid d} \left( 1 - \frac{2}{p} \right).$$ In particular, $$\frac{\psi(d)}{\Phi(d)} = \prod_{p \mid d} \frac{p-2}{p-1}.$$
\label{proposition: psi function invertibles modulo n}
\end{Proposition}

\begin{proof}
We start proving that $\psi$ is multiplicative, namely, if $m,n \in \NN$ are coprime, then $\psi(mn) = \psi(m) \psi(n)$. Indeed, by Chinese remainder theorem, the map
\begin{align*}
\varphi: \ZZ/mn\ZZ & \rightarrow (\ZZ/m\ZZ) \times (\ZZ/n \ZZ) \\
x + mn \ZZ & \mapsto (x + m \ZZ, x + n \ZZ)
\end{align*}
is a ring isomorphism. Thus, $1$ is sent to $(1,1)$ and units are sent to units. So $a-1$ is a unit in $\ZZ/mn\ZZ$ if and only if $\varphi(a) - 1 = (a_1-1,a_2-1)$ is a unit in $\ZZ/m\ZZ \times \ZZ/n\ZZ$, meaning that both $a_1-1$ and $a_2-1$ are units in their respective rings. Therefore, $\psi(mn) = \psi(m)\psi(n)$. 

We now compute $\psi(p^\alpha)$ for $p$ a prime number. The elements that are not units are of the form $pk$ with $k = 1,\dots,p^{\alpha-1}$, so $\psi(p^\alpha) = p^{\alpha-1}(p-2)$. 

Combining everything, if $d = \prod_{i = 1}^r p^{\alpha_i}$, then $$\psi(d) = \prod_{i = 1}^r p^{\alpha_i-1}(p-2) = d \prod_{i = 1}^r \left(1 - \frac{2}{p} \right).$$

Finally, since $\Phi(d) = d \prod_{i = 1}^r \left(1 - \frac{1}{p} \right)$, we obtain 
\begin{equation*}
\frac{\psi(d)}{\Phi(d)} = \prod_{p \mid d} \frac{p-2}{p-1}. \qedhere
\end{equation*}

\end{proof}

Thus, in the case where $I = (\ZZ/d\ZZ)^\times$, we have $\FPP(G_\mathcal{Q}^\mathcal{P}) = 0$ if and only if $d$ is even, since when $d$ is odd we have that $2$ and $2 -1 = 1$ are invertible modulo $\ZZ/d\ZZ$. 

\subsubsection{Construction 2}

Write $\mathcal{P} = \mathcal{Q} \rtimes \Aut(\mathcal{Q})$ with the operation $$(q,h)(q,h') = (q \, h(q'),h \, h').$$ In particular $(q,h)^{-1} = (h^{-1}(q^{-1}), h^{-1})$.

The group $\mathcal{P}$ acts on $\mathcal{P}/\Aut(\mathcal{Q})$ by left translation. A set of representatives for $\mathcal{P}/\Aut(\mathcal{Q})$ is $\set{(q,1)}_{q \in \mathcal{Q}}$.

\bigskip

The action has the following properties:

\begin{itemize}
\item \underline{The action is faithful:} Let $(q,h) \in \mathcal{P}$ such that $$(q,h)(q',1) \equiv (q',1) \pmod{\Aut(\mathcal{Q})}$$ for all $q' \in \mathcal{Q}$. Since $(q,h)(q',1) = (q \, h(q'), h)$, this means that $q \, h(q') = q'$ for every $q' \in \mathcal{Q}$. If $q' = 1$, then this implies that $q = 1$ and therefore $h(q') = q'$ for all $q'$, so $h$ is the identity. 

Thus, the action is faithful, and $\mathcal{P}$ can be seen as a subgroup of $\Sym(\abs{\mathcal{Q}})$. 

\item \underline{$\mathcal{Q}$ acts transitively:} Given $q_1,q_2 \in \mathcal{Q}$, then $(q_2 \, q_1^{-1},1) (q_1,1) = (q_2,1)$. 
\end{itemize}

\bigskip

The cosets of $\mathcal{P}/\mathcal{Q}$ have the form $$A_h = \set{(q,h): q \in \mathcal{Q}}.$$

An element $(q,h)$ fixes exactly one point in $\mathcal{P}/\Aut(\mathcal{Q})$ if and only if there exists a unique $q' \in \mathcal{Q}$ such that $$(q,h) (q',1) = (q \, h(q'), h) \equiv (q',1) \pmod{\Aut(\mathcal{Q})},$$ namely, $q \, h(q') = q'$ has a unique solution.

Given $d \geq 2$, write $d = 2^n r$ with $r$ odd and let $\mathcal{Q} = C_2^{n} \times C_r$, where $C_m$ is the cyclic group of order $m$. For this choice, $\Aut(\mathcal{Q}) \simeq \GL_n(\FF_2) \times (\ZZ/r\ZZ)^*$, where $\FF_2$ is the field with two elements.  

Elements $q \in \mathcal{Q}$ can be written as $(\vec{x},z)$, where $\vec{x} \in \FF_2^n$ and $z \in \ZZ/r\ZZ$, and elements $h \in \Aut(\mathcal{Q})$ as $(A,\alpha)$, where $A \in \GL_n(\FF_2)$ and $\alpha \in (\ZZ/r\ZZ)^*$. Since the group $\mathcal{Q}$ is abelian, we can use additive notation. Then $(q,h)$ has a unique fixed point in $\mathcal{P}/\Aut(\mathcal{Q})$ if and only if there exists a unique $q' = (\vec{x'}, z')$ such that $q + h(q') = q'$, namely, 
\begin{align*}
\left \{ \begin{matrix} \vec{x} + A \vec{x'} = \vec{x'}  \\ 
z + \alpha z' = z',  
\end{matrix}\right.
\end{align*}
or equivalently
\begin{align*}
\left \{ \begin{matrix} (A-1) \vec{x'} = -\vec{x}  \\ 
(\alpha-1) z' = -z.  
\end{matrix}\right.
\end{align*}

This will have a unique solution if and only if $A-1$ and $\alpha-1$ are invertible in $\GL_n(\FF_2)$ and $(\ZZ/r\ZZ)^*$, respectively, and this is independent of $q = (\vec{x}, z)$.

Therefore $$\FPP(G_\mathcal{Q}^\mathcal{P}) = \frac{\# \set{A \in \GL_n(\FF_2): A-1 \in \GL_n(\FF_2)}}{\abs{\GL_n(\FF_2)}} \, \prod_{p \mid r} \frac{p-2}{p-1}.$$

Notice that this is a generalization of construction 1 since when $d$ is odd, the group $$\set{x \mapsto x+b: b \in \ZZ/d\ZZ}$$ is isomorphic to $C_d$ and $$\set{x \mapsto ax+b: a \in (\ZZ/d\ZZ)^*, b \in \ZZ/d\ZZ}$$ is isomorphic to $C_d \rtimes \Aut(C_d) = C_d \rtimes (\ZZ/d\ZZ)^*$. The issue caused by $2$ when $d$ is even in construction 1 is, in this case, isolated. 

To prove that this construction of $G_\mathcal{Q}^\mathcal{P}$ is an example of a group with positive fixed-point proportion, we still need to find: 
\begin{enumerate}
\item a number in $\ZZ/r\ZZ$ such that $\alpha$ and $\alpha -1$ are invertible with $r$ odd and 
\item a matrix in $\GL_n(\FF_2)$ such that $A$ and $A-1$ are both invertible.
\end{enumerate}

In the case of $\alpha$, since $r$ is odd then $\alpha = 2$ satisfies the condition. 

For the matrices, define $$A_2 := \mat{1 & 1 \\ 1 & 0}$$ and $$A_3 := \mat{1 & 1 & 1 \\ 1 & 1 & 0 \\ 1 & 0 & 0}.$$

Since the rows are linearly independent, they are invertible. Moreover, $$A_2 - 1 = \mat{0 & 1 \\ 1 & 1}$$ and $$A_3 - 1 = \mat{0 & 1 & 1 \\ 1 & 0 & 0 \\ 1 & 0 & 1},$$ which are also invertible. Therefore, we have examples for $n = 2$ and $n = 3$. 

If $n > 3$, consider two cases. If $n$ is even, define $$A_n := \mat{A_2 & & & \\ & A_2 & & \\ & & \ddots & \\ & & & A_2 }$$ and if $n$ is odd, $$A_n := \mat{A_2 & & & & \\ & A_2 & & & \\ & & \ddots & & \\ & & & A_2 & \\ & & & & A_3}$$

By induction and the property of the determinant of subblocks, then $A_n$ and $A_n-1$ are invertible. 

The only case that cannot be addressed with this construction is when $n = 1$, or equivalently $d \equiv 2 \pmod 4$ because $\GL_1(\FF_2)$ has only one invertible element.

\cref{table: matrices invertible} shows the values of $$\# \set{A \in \GL_n(\FF_2): A-1 \in \GL_n(\FF_2)}$$ and $$\frac{\# \set{A \in \GL_n(\FF_2): A-1 \in \GL_n(\FF_2)}}{\abs{\GL_n(\FF_2)}}$$ for the first cases of $n$. Combined with \cref{proposition: psi function invertibles modulo n}, this allows us to calculate the fixed-point proportion for several cases of this construction.

\begin{table}[h!]
\begin{tabular}{|c|c|c|c|}
\hline
\textbf{$n$} & \textbf{$\# \set{A \in \GL_2(\FF_n): A-1 \in \GL_n(\FF_2)}$} & \textbf{$\abs{\GL_n(\FF_2)}$} & \textbf{$\frac{\# \set{A \in \GL_n(\FF_2): A-1 \in \GL_n(\FF_2)}}{\abs{\GL_n(\FF_2)}}$} \\ \hline
\textbf{1} & 0 & 1 & 0  \\ \hline
\textbf{2} & 2 & 6 & $1/3 \approx 0.333$ \\ \hline
\textbf{3} & 48 & 168 & $2/7 \approx 0.2857 \dots$  \\ \hline
\textbf{4} & 5824 & 20160 & $13/45 \approx 0.288 \dots$  \\ \hline
\textbf{5} & 2887680 & 9999360 & $188/651 \approx 0.2887 \dots$ \\ \hline
\end{tabular}
\caption{Number of matrices $A \in \GL_n(\FF_2)$ such that $A-I$ is also invertible.}
\label{table: matrices invertible}
\end{table}

There are other groups $\mathcal{Q}$ similar to $C_2^n \times C_r$ with $r$ odd that can be defined, and their fixed-point proportions will have a similar formula. For example, we can consider $\mathcal{Q} = \prod_{p \mid d } C_p^{e_p}$ where $p$ is a prime number and $e_p$ is the largest natural number such that $p^{e_p} \mid d$. 

\subsubsection{The problem of $d \equiv 2 \pmod{4}$}

Using GAP \cite{GAP}, all possible transitive groups $\mathcal{Q} \leq \Sym(d)$ were computed for $d \equiv 2 \pmod{4}$ with $d$ ranging from $2$ to $30$. Since $\mathcal{Q}$ must be normal in $\mathcal{P}$, we take $\mathcal{P}$ the normalizer of $\mathcal{Q}$ in $\Sym(d)$, since this corresponds to the most favorable scenario. However, for all tested values of $d$, no transitive group $\mathcal{Q}$ as subgroup of $\mathcal{P}$ had a coset where every element in the coset fixed exactly one element. This suggests an obstruction for these values of $d$ to admit a subgroup $\mathcal{Q}$ meeting the required conditions. The obstruction aligns with the special case considered in Construction 2, where $\mathcal{Q}$ was abelian of order $d$. The specific structure of $\mathcal{Q}$ in this case explains why the conditions fails for $d \equiv 2 \pmod{4}$.

\section{Application to iterated Galois groups}
\label{section: Application to iterated Galois groups}

In this subsection, we will demonstrate that Construction 1 in \cref{section: examples} corresponds to the iterated Galois group $G_\infty(K,x^d+c,t)$ where $K$ is any field, $t$ is transcendental over $K$ and $c \in K^\times$ such that $0$ is wandering for $f$ (namely, $f^n(0) \neq f^m(0)$ for all $n \neq m$).

\begin{proof}[Proof of \cref{theorem: FPP x^d+1}]

We start examining the solutions of $x^d + c = t$. These solutions are of the form $\xi^{x_1} \sqrt[d]{t-c}$ where $\xi$ is a primitive $d$-th root of unity, and $x_1 = 1,\dots,d$. Denote these solutions by $\alpha_{x_1}$. On the second level of the tree, we consider the preimages of each $\alpha_{x_1}$. These preimages are given by $$\alpha_{x_1 x_2} = \xi^{x_2} \sqrt[d]{\alpha_{x_1} - c}.$$ Continuing in this manner, at level $n$, the solutions are of the form $\alpha_{x_1 \dots x_n}$, where $$\alpha_{x_1 \dots x_n} = \xi^{x_n} \alpha_{x_1 \dots x_{n-1}d}$$ and $$\alpha_{x_1 \dots x_{n-1}d} = \sqrt[d]{\alpha_{x_1\dots x_{n-1}} - c}.$$ In the $d$-regular tree, we place $\alpha_{x_1\dots x_n}$ in the $n$-th level of the tree, ordered by lexicographical order. Notice that all preimages in the same level are different by \cref{lemma: t transcendental d^n preimages}.

To calculate the Galois group $G_\infty(K,x^d+c,t)$, we rely on the following theorem:

\begin{theorem}[{\cite[Theorem 3.1]{Juul2014}}]
Let $k$ be a field and $f \in k(x)$ a rational function with degree $d \geq 2$ such that $f' \neq 0$. Let $t$ be transcendental over $k$, write $L = k(t)$ and $L_\infty =  L \left( \bigcup_{n \in \NN} f^{-n}(t) \right)$. Let $C$ denote the set of critical points of $f$, write $F = L(f^{-1}(t)) \cap \overline{k}$ and $\mathcal{Q} = \Gal(L(f^{-1}(t))/F(t))$. If for any $a,b \in C$ and $m,n \in \NN$ we have $f^m(a) = f^n(b)$ if and only if $m = n$ and $a = b$, then $\Gal(L_\infty/F(t)) \simeq W_\mathcal{Q}$, where $W_\mathcal{Q}$ is the iterated wreath product of $\mathcal{Q}$.
\label{teo_GaloisGroup_wreathproduct_TuckerJuul}
\end{theorem}

In our case, the polynomial $f$ has only $0$ as its critical point, and the orbit of $0$ is wandering by the choice of $c$. Thus, $\Gal(K_\infty(f)/F(t)) = W_\mathcal{Q}$, where $F = K_1(f) \cap \overline{K} = K(\xi)$, and $$\mathcal{Q} = \Gal(K_1(f)/F(t)) = \Gal(F(\alpha_d,t)/F(t)).$$ If $\sigma \in \mathcal{Q}$, there exists $b \in \ZZ/d\ZZ$ such that $\sigma(\alpha_d) = \alpha_b$ and then $$\sigma(\alpha_x) = \sigma(\xi^x \alpha_d) = \xi^x \alpha_b = \alpha_{x + b}.$$ This implies that $\mathcal{Q} \simeq \set{x \mapsto x+b: b \in \ZZ/d \ZZ}$.

Now, consider the monomorphism of groups $\varphi: \Gal(K(\xi)/K) \rightarrow (\ZZ/d\ZZ)^\times$ given by $\sigma \mapsto a$ such that $\sigma(\xi) = \xi^a$ and write $I = \IIm(\varphi)$.

If $g \in G_\infty(K,x^d+c,t)$, then there exists $a \in I$ such that $\sigma(\xi) = \xi^a$ and given $x_1,\dots, x_{n-1} \in \ZZ/d\ZZ$, there exist $y_1,\dots, y_n \in \ZZ/d\ZZ$ such that $\sigma(\alpha_{x_1 \dots x_{n-1}d}) = \alpha_{y_1 \dots y_n}$. Then, if $x_n \in \ZZ/d\ZZ$,
\begin{align*}
g(\alpha_{x_1 \dots x_{n-1}x_n}) = g(\xi^{x_n} \alpha_{x_1 \dots x_{n-1}d}) = g(\xi)^{x_n} \alpha_{y_1 \dots y_n} = \\ \xi^{a x_n} \alpha_{y_1 \dots y_n} = \alpha_{y_1 \dots y_{n-1}(ax_n + y_n)}.
\end{align*}

This shows that the label $g|_{x_1 \dots x_{n-1}}^1 = x \mapsto ax + y_n$. Notice that the slope of the linear action of the label does not depend on the vertex and it only depends on the element $g$. Therefore, if we call $\mathcal{P} = \set{x \mapsto ax+b: a \in I, b \in \ZZ/d\ZZ}$, we obtain that $G_\infty(K,x^d+c,t) \leq G_\mathcal{Q}^\mathcal{P}$. 

If we take the quotient $G_\infty(K, x^d+c,t)/W_\mathcal{Q}$, we have 
\begin{align*}
G_\infty(K, x^d+c,t)/W_\mathcal{Q} = G_\infty(K, x^d+c,t)/\Gal(K_\infty(f)/F(t)) \\ \simeq \Gal(K(\xi,t)/K(t)) \simeq \Gal(K(\xi)/K(t)) \simeq I
\end{align*}
and so $$[G_\infty(K, x^d+c,t):W_\mathcal{Q}] = \abs{I} = \abs{\mathcal{P}/\mathcal{Q}} = [G_\mathcal{Q}^\mathcal{P}: W_\mathcal{Q}].$$ 
Thus, $G_\infty(\QQ,f,t) = G_\mathcal{Q}^\mathcal{P}$ and $\FPP(G_\infty(\QQ,f,t))$ is given by \cref{equation: FPP Galois x^d+1}. 
\end{proof}

\section{Open questions}
\label{section: open questions}

We conclude this article by posing the following questions, which are considered by the author to be relevant and interesting for a better understanding of the theory of the fixed-point proportion of groups acting on trees:

\begin{enumerate}
\item \cref{lemma: continuity FPP} answers the question about the continuity of the function $\FPP$. The next question is about the surjectivity of $\FPP$ in $[0,1]$:

\begin{Question}
Fix $d \in \NN$ and $T$ a $d$-regular tree. Given $\alpha \in [0,1]$, can we find a group $G \leq \Aut(T)$ such that $\FPP(G) = \alpha$? Does the same hold if we restrict to topologically finitely generated groups?
\label{question: FPP surjective}
\end{Question}

If we restrict to level-transitive groups, the answer to \cref{question: FPP surjective} is no. Indeed, as the fixed-point proportion has to be smaller than the proportion on the first level by the proof of \cref{lemma: FPP first properties}, if $F_1$ denotes the set of elements in $\pi_1(G)$ fixing at least one vertex in $\mathcal{L}_1$, by Burnside's lemma,
\begin{align*}
1 = \frac{1}{\abs{\pi_1(G)}} \sum_{g \in \pi_1(G)} \# (\mathcal{L}_1^g) = \frac{1}{\abs{\pi_1(G)}} \sum_{g \in F_1} \# (\mathcal{L}_1^g) \\ = \frac{1}{\abs{\pi_1(G)}} \left(d + \sum_{g \in F_1 \setminus \set{1}} \# (\mathcal{L}_1^g) \right) \geq \frac{1}{\abs{\pi_1(G)}} \left(d + \# F_1 -1 \right) \\ = \frac{d-1}{\abs{\pi_1(G)}} + \frac{\# F_1}{\abs{\pi_1(G)}} 
\end{align*}
Solving for the last term and using that $\pi_1(G) \leq \Sym(d)$,
\begin{align*}
\FPP(G) \leq \frac{\# F_1}{\pi_1(G)} = 1 - \frac{d-1}{\abs{\pi_1(G)}} \leq 1 - \frac{d-1}{d!}.
\end{align*}
For example, when $d = 2$, we obtain that no level-transitive group can have fixed-point proportion bigger than $1/2$. Moreover, doing the same argument in the second level and inspecting the possible transitive groups $\pi_2(G)$, it can be seen that for level-transitive groups in the binary tree, the fixed-point proportion cannot be bigger than $3/8$.

This leads us to rephrase \cref{question: FPP surjective} in the case of the family of level-transitive groups:

\begin{Question}
Fix $d \in \NN$ and $T$ a $d$-regular tree. What is the maximum fixed-point proportion for level-transitive groups acting on $T$? Let $\alpha_d$ be this maximum. Given $\alpha \in [0,\alpha_d]$, does there exist $G$ a level-transitive group of $\Aut(T)$ such that $\FPP(G) = \alpha$? 
\end{Question}

\item Finally, we wonder about the behavior of the fixed-point proportion when we take random subgroups. 

\begin{Question}
Let $T$ be a spherically homogeneous tree, $n$ a positive integer, and $G \leq \Aut(T)$ a closed subgroup with normalized Haar measure $\mu$ and $\FPP(G) = 0$. Consider $G^n$ as a probability space equipped with the product measure. Select $n$ random elements of $G$ and let $H$ the subgroup they generate. Does $\FPP(H) = 0$ with probability $1$?
\label{question: FPP random subgroups}
\end{Question}
\end{enumerate}

\bibliographystyle{unsrt}

\end{document}